\documentclass[12pt,reqno]{amsart}
\usepackage{amsmath}
\usepackage{amssymb}
\usepackage{amstext}
\usepackage{a4wide}
\usepackage{graphicx}
\allowdisplaybreaks \numberwithin{equation}{section}
\usepackage{color}
\usepackage{cases}

\numberwithin{equation}{section}

\newtheorem{theorem}{Theorem}[section]
\newtheorem{proposition}[theorem]{Proposition}

\newtheorem{lemma}[theorem]{Lemma}
\newtheorem*{Yudovich's Theorem}{Yudovich's Theorem}
\newtheorem*{lemmaA}{Lemma A}

\theoremstyle{definition}

\newtheorem{definition}[theorem]{Definition}

\theoremstyle{remark}
\newtheorem{remark}[theorem]{Remark}

\begin{document}

\title
[Stability of two-dimensional steady Euler flows]{Stability of two-dimensional steady Euler flows with concentrated vorticity}

 \author{Guodong Wang}

\address{Institute for Advanced Study in Mathematics, Harbin Institute of Technology, Harbin 150001, P.R. China}
\email{wangguodong@hit.edu.cn}


\begin{abstract}

In this paper, we study the stability of two-dimensional steady Euler flows  with sharply concentrated vorticity in a  bounded domain. These flows are  obtained as maximizers of the kinetic energy on some isovortical surface, under the constraint that the vorticity
is compactly supported in a finite number of disjoint regions of small diameter.
We prove the nonlinear stability of  these flows when the vorticity is sufficiently concentrated in one small region, or in two small regions  with opposite signs.
 The proof is achieved by showing that these  flows   constitute a compact and isolated set of local maximizers of the kinetic energy on the isovortical surface. The separation property of the stream function plays a crucial role in validating   isolatedness.
 \end{abstract}

\maketitle

\section{Introduction and main results}

\subsection{Two-dimensional Euler equation}
Let $D\subset \mathbb R^2$ be a smooth bounded domain.
Consider the following Euler system governing the motion of an incompressible inviscid  fluid of unit density in $D$
\begin{equation}\label{1-1}
\begin{cases}
\partial_t\mathbf{v}+(\mathbf{v}\cdot\nabla)\mathbf{v}=-\nabla P,&x=(x^1,x^2)\in D,\,\, t>0,\\
\nabla\cdot\mathbf{v}=0,
 \end{cases}
\end{equation}
where $\mathbf{v}=(v^1,v^2)$ is the velocity field and $P$ is the pressure. The scalar vorticity of the fluid is defined as
\[\omega=\partial_{x^1}v^2-\partial_{x^2}v^1.\]
From the momentum equation (i.e., the first equation of \eqref{1-1}), we can easily get the following evolution equation for $\omega$
\begin{equation}\label{dxx}
\partial_t\omega+\mathbf v\cdot\nabla \omega=0.
\end{equation}
Under suitable assumptions $\mathbf v$ can be recovered from $\omega$ via the Biot-Savart law.
For example, if $D$ is additionally simply-connected and the following impermeability boundary condition holds
\begin{equation}\label{dxx0}
\mathbf v\cdot\mathbf n=0,\quad x\in\partial D,
\end{equation}
where $\mathbf n$ is the outward unit normal to $\partial D,$ then the Biot-Savart law can be expressed as
\begin{equation}\label{dxx1}
\mathbf v=\nabla^\perp \mathcal G\omega.
\end{equation}
Here $\mathcal G$ is the Green operator corresponding to  $-\Delta$ in $D$ with zero boundary condition, or equivalently,
\begin{equation}\label{gpo}
\begin{cases}
-\Delta \mathcal G\omega=\omega,&x\in D,\\
\mathcal G\omega=0, &x\in\partial D,
 \end{cases}
\end{equation}
and $\nabla^\perp\mathcal G\omega$ is the clockwise rotation through $\pi/2$ of $\nabla\mathcal G\omega,$ that is,
$$\nabla^\perp\mathcal G\omega=(\partial_{x^2}\mathcal G\omega,-\partial_{x^1}\mathcal G\omega).$$
The function $\mathcal G\omega$ is called \emph{the stream function} related to $\omega$.
If $D$ is not simply-connected, or $\mathbf v$ is not tangential to the boundary, then the Biot-Savart law has a more complicated form than \eqref{dxx1}.  For simplicity, \emph{throughout this paper we assume that $D$ is simply-connected and \eqref{dxx0} holds}. Hence the evolution equation of $\omega$ is
\begin{equation}\label{dxx6}
\partial_t\omega+\nabla^\perp\mathcal G\omega\cdot\nabla \omega=0,
\end{equation}
which is usually called the vorticity equation. In the rest of this paper we   focus our main attention on \eqref{dxx6}.

In the literature, there are a lot of existence results on the initial-value problem of  \eqref{dxx6} with initial vorticity in various function spaces, including
those by H\"older \cite{Ho} and Wolibner \cite{Wo} in H\"older spaces, Yudovich \cite{Y} in $L^\infty,$ DiPerna-Majda \cite{DM} in $L^p$ with $ 1<p<+\infty$ (see also the independent work Giga-Miyakawa-Osada \cite{Giga}), and Delort \cite{De} in the space of nonnegative Radon measures in $H^{-1}.$   In this paper we only work in Yudovich's setting, that is, the vorticity belongs to $L^\infty$. The reasons are threefold. First, it contains enough solutions that are physically interesting, such as flows with discontinuous vorticity. Second, uniqueness holds in this case, which makes many statements concise. Third, two-dimensional Euler flows with bounded vorticity possess some good conservative properties that are crucial  in proving stability.

Before stating Yudovich's result, we list some notations that will be used throughout this paper.
\begin{itemize}
\item $\mathcal L:$ the two-dimensional Lebesgue measure;
\item $\mathcal R_f$: the rearrangement class of some function $f\in L^1_{\rm loc}(D)$, that is,
 \begin{equation*}
 \mathcal{R}_f=\left\{ g \in L^1_{\rm loc}(D)\mid  \mathcal L(\{x\in D\mid g(x)>s\})= \mathcal L(\{x\in D\mid f(x)>s\}) ,\,\,\forall \,s\in \mathbb R\right\};
 \end{equation*}
\item ${\rm supp}(f)$: the essential support of some measurable function $f$ (see  \S 1.5, \cite{LL} for the precise definition);
\item $L_c^\infty(\mathbb R^2)$: the set of essentially bounded functions in $\mathbb R^2$ with compact support;
\item sgn$(a)$: the sign of some real number $a$, that is,
\begin{equation*}
{\rm sgn}(a)=
\begin{cases}
-1,& \mbox{if }a<0,\\
0,&\mbox{if } a=0,\\
1,&\mbox{if } a>0;
\end{cases}
\end{equation*}
\item $\mathbf 1_A$: characteristic function of some set $A$;
\item $\mathbf 0$: the origin in $\mathbb R^2;$
\item $B_r(x):$ the disk with radius $r$ and center $x$;
\item $f_+(f_-)$: the positive (negative) part of $f,$ that is, $f_+=\max\{f,0\}$ ($f_-=\max\{-f,0\}$);
\item $G$: the Green function of $-\Delta$ in $D$ with zero boundary condition;
\item $h$: the regular part of the Green function $G$, that is,
\[h(x, y)=-\frac{1}{2\pi}\ln|x-y|-G(x,y),\quad  x, y\in D;\]
\item $H$: the Robin function of the domain $D$, defined by
\begin{equation}\label{ropz89}
H(x)=h(x,x),\quad x\in D.
\end{equation}
\end{itemize}

Yudovich's result can be stated as follows.
 \begin{Yudovich's Theorem}\label{A}
 Let $\omega_0\in L^\infty(D)$. Then there exists a unique weak solution $\omega \in L^\infty(D\times(0,+\infty))$ to the vorticity equation \eqref{dxx6} in the following sense
\begin{equation}\label{tweak}
  \int_D\omega_0(x)\phi(x,0)dx+\int_0^{+\infty}\int_D\omega\partial_t\phi+\omega\nabla\phi\cdot \nabla^\perp \mathcal G\omega dxdt=0,\,\,\forall\,\phi\in C_c^{\infty}(D\times\mathbb R).
  \end{equation}
 Moreover, such a solution $\omega$ satisfies
\begin{itemize}
\item[(i)] $\omega(\cdot,t)\in \mathcal{R}_{\omega_0}$ for any $t\geq 0$;
\item[(ii)] $\omega\in C([0,+\infty);L^p(D))$ for any $p\in[1,+\infty);$
\item[(iii)]  $E(\omega(\cdot,t))=E(\omega_0)$ for any $t\in[0,+\infty),$
where
\begin{equation}\label{tweak2}E(\omega(\cdot,t))=\frac{1}{2}\int_D|\mathbf v|^2 dx=\frac{1}{2}\int_D\omega\mathcal G\omega dx
\end{equation}
 is the kinetic energy of the fluid.
\end{itemize}

 \end{Yudovich's Theorem}

 See Majda-Bertozzi \cite{MB} or Marchioro-Pulvirenti \cite{MP4} for the detailed proof of Yudovich's Theorem.

 \begin{remark}
In Yudovich's Theorem, for fixed $t\geq0$, since $\omega(\cdot,t )\in L^{\infty}(D)$, we can deduce from standard elliptic estimates that $\mathcal{G}\omega(\cdot,t)\in C^{1}(\bar{D})$, therefore the integrals in \eqref{tweak} and \eqref{tweak2} make sense.
\end{remark}

 \begin{remark}
In Yudovich's Theorem, by \eqref{tweak} and (ii)  it is easy to check that
\[\lim_{t\to0^+}\|\omega(\cdot,t)-\omega_0\|_{L^p(D)}=0\]
for any $p\in[1,+\infty).$
\end{remark}

By Yudovich's Theorem, an ideal flow evolves on some ``isovortical" surface with the kinetic energy unchanged. Here by an isovortical surface we mean a set of all instantaneous incompressible flows whose vorticities constitute the rearrangement class of a given measurable function. As we will see in Section 3, this fact plays an important role in studying the stability of two-dimensional steady Euler flows.

\subsection{Steady solution and stability}
A steady solution to the vorticity equation is a solution that does not depend on the time variable. Hence a steady solution $\omega$ satisfies
\begin{align}\label{sv1}
  \nabla^\perp \mathcal{G}\omega\cdot\nabla\omega=0,\,\,x\in D.
\end{align}
For $\omega$ only in $ L^\infty(D)$, \eqref{sv1} should be interpreted in the following weak sense.
\begin{definition}\label{wsv1}
Let $\omega\in L^{\infty}(D)$. We call  $\omega$ a steady weak solution to the vorticity equation if it satisfies
\begin{align}\label{int11}
  \int_D\omega\nabla^\perp\mathcal{G}\omega \cdot\nabla \zeta dx=0,\quad \forall \,\zeta\in C_c^\infty(D).
\end{align}
\end{definition}
It is easy to check that Definition \ref{wsv1} is consistent with \eqref{tweak}.

Without any restriction there are infinitely many steady weak solutions. For instance, any $\omega\in C^1(\bar D)$ satisfying
\[\omega=f(\mathcal G\omega) \,\,\,{\rm in }\,\,D,\]
where $f\in C^1(\mathbb R)$, must satisfy \eqref{int11}. A more general criterion has been proved in \cite{CW3}.

\begin{lemmaA}[\cite{CW3}, Theorem 1.2]\label{lem200}
Let $k$ be a positive integer.
Suppose $\omega\in L^{\infty}(D)$ satisfies
\begin{equation}\label{cdan}
\omega=\sum_{i=1}^k\omega_i, \,\,\min_{1\leq i< j\leq k}\{{\rm dist(supp(}\omega_i),{\rm supp(}\omega_j))\}>0,\,\,\omega_i=f_i(\mathcal G\omega) \text{  \rm   a.e. in } {\rm supp(}\omega_i)_\delta
\end{equation}
for some $\delta>0$, where  supp$(\omega_i)_\delta$ is the $\delta$-neighborhood of  supp$(\omega_i)$ in the Euclidean norm, that is,
 \[{\rm supp(}\omega_i)_\delta=\{x\in D\mid {\rm dist(}x,{\rm supp(}\omega_i))<\delta\},\]
 and each $f_i:\mathbb R\to\mathbb R$ is either monotone or locally Lipschitz continuous.
Then $\omega$ is a steady weak solution to the  vorticity equation.
\end{lemmaA}
Lemma A will be used in the proof of Theorem \ref{thmex} in Section 6.

Given a steady solution, an interesting and important problem to study is its stability. In this paper, we only consider stability of Lyapunov type, also called nonlinear stability.
To make it general, we give the definition of stability for a set of steady weak solutions.
\begin{definition}\label{sdef}
Let $\mathcal M \subset L^\infty(D)$ be a nonempty set of steady weak solutions to the vorticity equation, $\|\cdot\|$ be a norm on $L^\infty(D),$ $\mathcal P\subset L^\infty(D)$ be nonempty. If for any $\varepsilon>0,$ there exists $\delta>0$, such that for any $\omega_0\in\mathcal P$ satisfying
\[\inf_{w\in\mathcal M}\|w-\omega_0\|<\delta,\]
it holds that
\[\inf_{w\in\mathcal M}\|w-\omega(\cdot, t)\|<\varepsilon,\quad\forall\,t\geq 0,\]
where $\omega$ is the weak solution to the vorticity equation with initial vorticity $\omega_0,$ then $\mathcal M$ is said to be stable in the norm $\|\cdot\|$ with respect to initial perturbations in $\mathcal P.$
\end{definition}
Sometimes the stability in Definition \ref{sdef} is also called  orbital stability. When
 $\mathcal M$ contains only one element $\bar\omega$ and is stable, we say $\bar\omega$ is stable.

Commonly used norms include the $L^p$ norm of the vorticity $\|\omega\|_{L^p(D)}$, $1\leq p<+\infty$, and the energy norm
\[\|\omega\|_E=\left(\int_D|\nabla \mathcal G\omega|^2dx\right)^{1/2}.\]

\subsection{Steady Euler flows with concentrated vorticity}
In many natural phenomena such as tornados, the vorticity is sharply concentrated in a finite number of small regions and  vanishes elsewhere.  Mathematically, the vorticity should be assumed to possess the following form
\begin{equation}\label{cv}
\omega^\varepsilon=\sum_{i=1}^k\omega^{\varepsilon}_i,\quad
{\rm supp}(\omega^{\varepsilon}_i)\subset B_{o(1)}( x_i),\quad\int_D\omega^{\varepsilon}_i dx=\kappa_i+o(1),\quad i=1,\cdot\cdot\cdot,k,
\end{equation}
where $\varepsilon$ is a small positive parameter, $k$ is a positive integer, $x_1,\cdot\cdot\cdot,x_k$ are $k$ different points in $D$, $\kappa_1,\cdot\cdot\cdot,\kappa_k$ are $k$ nonzero real numbers, and $o(1)\to0$ as $\varepsilon \to0$.
The study for such kind of flows is rather difficult in both theoretical and numerical analysis because of the presence of strong singularity.
To simplify the problem, a related finite dimensional model, called the point vortex model, has been introduced. See  \cite{L} for example. In the point vortex model, each $\omega^{\varepsilon}_i$ is replaced by a Dirac measure located at $x_i$, and the evolution of  each $x_i$ is described by the following system of ordinary differential equations
\begin{equation}\label{pvml}
\kappa_i\frac{dx_i}{dt}=-\nabla^\perp_{x_i} W(x_1,\cdot\cdot\cdot,x_k),\quad i=1,\cdot\cdot\cdot,k,
\end{equation}
where $W$ is the Kirchhoff-Routh function related to $\kappa_1,\cdot\cdot\cdot,\kappa_k$, this is,
\begin{equation}\label{krf}
W(x_1,\cdot\cdot\cdot,x_k)=-\sum_{1\leq m<n\leq k}\kappa_m\kappa_nG(x_m, x_n)+\frac{1}{2}\sum_{n=1}^k\kappa_n^2h( x_n,x_n),\,\,( x_1,\cdot\cdot\cdot,x_k)\in\mathbb D^k,
\end{equation}
where
\[\mathbb D^k=\underbrace{D\times\cdot\cdot\cdot\times D}_{k\text { times}}\setminus\left\{(x_1,\cdot\cdot\cdot, x_k)\mid x_m\in D, x_m= x_n \text{ for some }m\neq n\right\}.\]
The point vortex model is only an approximate model, and rigorous analysis on its connection with the vorticity equation with concentrated vorticity is  an interesting research topic. We refer the interested readers to  \cite{DI, M,MP1,MP2,MP3,MP4,MPa,T3} for some deep results in this respect.

In this paper, we are mainly concerned with \emph{steady  flows} with concentrated vorticity.
From the point vortex model,
the locations of concentrated blobs of vorticity should constitute a critical point of $W$. This has been proved rigorously in  \cite{WZ} by using the anti-symmetry of the Biot-Savart kernel.
The inverse problem is: given a critical point $(x_1,\cdot\cdot\cdot,x_k)$ of $W$,  can we construct a family of steady Euler flows such that the vorticities ``shrink" to these $k$ points?  This problem has been studied by many authors in the past several decades and  lots of solutions of the form \eqref{cv} have been obtained via various methods.
See \cite{CLW,CPY,CWZ, EM0, EM, SV,T} and the references therein.

Roughly speaking, these existence results can be classified into two different types.
For the first type,  the rearrangement of the vorticity is prescribed, and the vorticity has the form  \eqref{cdan}; however, the  profile function $f_i$ is unknown. See \cite{EM0, EM} for example. We will discuss this type of solutions in detail below. For the second type, the vorticity has the form \eqref{cdan} with each $f_i$ prescribed, but the rearrangement of the vorticity is unknown. See \cite{CLW, CPY2, CWZ,SV} for example. Of course, for some special kind of solutions such as vortex patches, both  the rearrangement of the vorticity and the profile functions are known. See \cite{CPY, T}.
Our stability result (i.e.,, Theorem \ref{el2k2} below) is about solutions of the first type; however,  we will show in the last section that our method can also be used to handle the stability of the second type of flows under certain circumstance.

Our first result is to about the existence of a large class of concentrated steady vortex flows of the first type.  It is achieved by studying the following variational problem.

Let $k$ be a positive integer, $\vec\kappa=(\kappa_1,\cdot\cdot\cdot,\kappa_k),$ where each $\kappa_i$ is a nonzero real number. Let $(\bar x_1,\cdot\cdot\cdot,\bar x_k)\in\mathbb D^k$ be an isolated local minimum point of the Kirchhoff-Routh function $W$ related to $\vec\kappa$, that is,
\begin{equation*}
{ W}(x_1,\cdot\cdot\cdot,x_k)=- \sum_{1\leq i<j\leq k}\kappa_i\kappa_jG(x_i,x_j)+ \frac{1}{2}\sum_{i=1}^k\kappa_i^2h(x_i,x_i).
\end{equation*}
Choose $\bar r>0$ sufficiently small such that
\begin{equation}\label{br1}
\overline{B_{\bar r}(\bar x_i)}\subset D,\quad\forall\,1\leq i\leq k,
\end{equation}
\begin{equation}\label{br2}
\overline{B_{\bar r}(\bar x_i)}\cap \overline{B_{\bar r}(\bar x_j)}=\varnothing,\quad\forall\,1\leq i< j\leq k,
\end{equation}
\begin{equation}\label{br3}
 (\bar x_1,\cdot\cdot\cdot,\bar x_k) \mbox{ is the unique minimum point of $W$ in } \overline{B_{\bar r}(\bar x_1)}\times\cdot\cdot\cdot\times \overline{B_{\bar r}(\bar x_k)},
\end{equation}
where $\overline{B_{\bar r}(\bar x_i)}$ is the closure of ${B_{\bar r}(\bar x_i)}$ in the Euclidean topology.

Define
\begin{equation}\label{jjk}
\Xi=\{\xi\in L_c^\infty(\mathbb R^2)\mid \xi \geq 0 \mbox{ a.e.},\,\,  \xi \mbox{ is radially symmetric and nonincreasing} \}.
\end{equation}
Let $M$ be a fixed positive number, $\Pi^1,\cdot\cdot\cdot,\Pi^k:(0,\bar r)\to\Xi$ be $k$ maps such that for each $i\in\{1,\cdot\cdot\cdot,k\}$,
 \begin{equation}\label{h13}
 {\mbox{supp($\Pi^i_\varepsilon$)}}\subset \overline{B_\varepsilon(\mathbf 0)},
\quad \forall\, \varepsilon\in(0,\bar r),
\end{equation}
 \begin{equation}\label{hh13}
 \lim_{\varepsilon\to0}\int_{\mathbb R^2}\Pi^i_\varepsilon (x)dx=|\kappa_i|,
\end{equation}
 \begin{equation}\label{hhhh13}
  \|\Pi^i_\varepsilon\|_{L^\infty(\mathbb R^2)}\leq M\varepsilon^{-2}.
 \end{equation}
 Here and henceforth, we denote $\Pi^i_\varepsilon=\Pi^i(\varepsilon)$ for simplicity. For example,  we can choose
\[\Pi^i_\varepsilon(x)=\frac{1}{\varepsilon^2}\varrho_i\left(\frac{x}{\varepsilon}\right),\quad i=1,\cdot\cdot\cdot,k,\]
where $\varrho_1,\cdot\cdot\cdot,\varrho_k\in \Xi$ are $k$ fixed functions such that
\[\mbox{supp}(\varrho_i)\subset \overline{B_1(\mathbf 0)},\quad \|\varrho_i\|_{L^1(\mathbb R^2)}=|\kappa_i|,\quad \|\varrho_i\|_{L^\infty(\mathbb R^2)}\leq M,\quad \forall\,i=1,\cdot\cdot\cdot,k.\]

Define
\begin{equation}\label{kvar}
\mathcal K_{\varepsilon}=\left\{w=\sum_{i=1}^k w_i\,\,\bigg|\,\, w_i \mbox{ is a rearrangement of  sgn}(\kappa_i)\Pi^i_\varepsilon, \,\, {\rm supp}(w_i)\subset \overline{B_{\bar r}(\bar x_i)}\right\}.
\end{equation}
Denote $\mathcal M_{\varepsilon}$ the set of maximizers of $E$ over $\mathcal K_{\varepsilon}$.

Our first result can be stated as follows.
\begin{theorem}[Existence]\label{thmex}
For any $\varepsilon\in(0,\bar r)$,  $\mathcal M_{\varepsilon}$ is not empty, and the following assertions hold.
\begin{itemize}
\item[(i)] There exists some $\varepsilon_1>0$, such that for any $\varepsilon\in(0,\varepsilon_1)$,  $ \mathcal M_{\varepsilon}$ is a set of  steady solutions to the  vorticity  equation  in the sense of Definition \ref{wsv1}.
\item [(ii)] For each $i\in\{1,\cdot\cdot\cdot,k\}$, $\mbox{ supp}(\omega\mathbf 1_{B_{\bar r}(\bar x_i)})$ ``shrinks" to $\bar x_i$  uniformly as $\varepsilon\to0$. More precisely,
for any $\delta\in(0,\bar r),$ there exists some $\varepsilon_2>0,$  such that for any $\varepsilon\in(0,\varepsilon_2)$, it holds that
\[\mbox{\rm supp}(\omega\mathbf 1_{B_{\bar r}(\bar x_i)})\subset \overline{B_{\delta}(\bar x_i)},\quad\forall\,\omega \in\mathcal M_{\varepsilon}.\]
\end{itemize}
\end{theorem}

The proof of Theorem \ref{thmex}  will be provided  in Section 6.

\subsection{Stability result}

Now we are ready to state our second result concerning the stability of $\mathcal M_\varepsilon$.

\begin{theorem}[Stability]\label{el2k2}
Fix $1<p<+\infty.$
Suppose $k=1$ or $k=2$ with $\kappa_1\kappa_2<0$. Then there exists some $\varepsilon_3>0,$ such that for any $\varepsilon\in(0,\varepsilon_3)$,  $\mathcal M_{\varepsilon}$  is  stable in the $L^p$ norm of the vorticity with initial perturbations in $L^\infty(D)$.
More precisely, for any $\varepsilon\in(0,\varepsilon_3)$ and $\epsilon>0$, there exists some $\delta>0,$ such that
 for any $\omega_0\in L^\infty(D)$,
\[ \inf_{w\in\mathcal M_\varepsilon}\|w-\omega_0\|_{L^p(D)}<\delta\Longrightarrow \inf_{w\in\mathcal M_\varepsilon}\|w-\omega(\cdot, t)\|_{L^p(D)}<\epsilon\,\,\,\forall\,t\geq 0,\]
where $\omega(x,t)$ is the unique weak solution to the vorticity equation with initial vorticity $\omega_0.$

\end{theorem}

If we choose $\Pi^i$ as follows:
\[\Pi^i_\varepsilon=\frac{\kappa_i}{\pi\varepsilon^2}\mathbf 1_{B_{\varepsilon}(\mathbf 0)},\quad i=1,\cdot\cdot\cdot,k,\,\,\varepsilon\in(0,\bar r),\]
then $\mathcal M_\varepsilon$ is a set of steady vortex patches, and the corresponding stability results have been obtained in \cite{CW2} for $k=1$ and \cite{CW1} for $k=2, \kappa_1\kappa_2<0$. However, the methods in \cite{CW1,CW2} rely heavily on the facts that the boundary of a concentrated steady vortex patch is $C^1$ and the stream function is nondegenerate on it.
For general $\Pi_\varepsilon^i$, there  is  little hope to apply the methods in \cite{CW1,CW2} to obtain stability.

To prove Theorem \ref{el2k2}, we introduce some new ideas. First, we prove a general stability criterion, i.e., Theorem \ref{thm551} in Section 3, showing that an isolated and compact set of local maximizers of the kinetic on some isovortical surface must be stable. This generalizes Burton's stability criterion  in \cite{B5} for a single isolated local maximizer. From this criterion, we need only to verify the compactness and isolatedness of $\mathcal M_{\varepsilon}$. Compactness is easy, and can be obtained in view of the variational nature of $\mathcal M_{\varepsilon}$.  The trouble is isolatedness.  Our second innovation is that we develop a new technique to prove the  isolatedness of  $\mathcal M_{\varepsilon}$.  We
discover that the stream function has a very good property, which we call the separation property, when the vorticity is sufficiently  concentrated. With the separation property, together with the symmetry and positivity of the Green operator, we can employ a lemma (i.e., Lemma \ref{lem202} in Section 2) from measure theory to obtain isolatedness.

It is worth mentioning that our method can also be used to handle
 the stability of some other concentrated vortex flows. An example is given in Section 7.

Before ending this section, we recall some related works on the stability of two-dimensional steady Euler flows and make some comparisons.
The study for the stability of steady Euler flows in two dimensions has a long history, dating back to the works of Kelvin \cite{K} and Love \cite{Lo}. The first systematic and general method of proving nonlinear stability was established by Arnol'd in 1960s. Arnol'd in \cite{A1,A2} proposed  the well-known energy-Casimir (EC) functional method and used it to prove what are now usually called Arnol'd's first and second stability theorems. In 1990s,  Wolansky-Ghil \cite{WG0,WG} developed the EC functional method by introducing the idea of supporting functionals and obtained some extensions of Arnol'd's second stability theorem. See also \cite{LZ2, W2}.  Although Arnol'd's stability theorems and their extensions are powerful tools in  studying the nonlinear stability of steady Euler flows, they have strong limitations as well. First, they can only deal with flows with vorticity satisfying
\[\omega=f(\mathcal G\omega)\,\,\mbox{a.e. in }D,\]
where the profile function $f$ is given. However, in Section 6 we will see that the profile functions related to the flows in Theorem \ref{el2k2} are unknown. Second, they usually require $f$ to be $C^1$, which excludes the possibility of application to many interesting cases such as vortex patches. Third,
these criteria have strong requirements on the energy of solutions. For example, Arnol'd's second stability requires the operator $-\Delta-f'(\mathcal G\omega)$ in $L^2(D)$ to be positive, which does not hold or is hard to verify in most cases.  For these reasons, it is almost impossible to prove Theorem \ref{el2k2} by simply applying these general stability criteria.

Except for Arnol'd's stability theorems and their extensions, there are also some stability results about special steady Euler flows in the literature.
See \cite{CW1,CW2, Ta,WP} about steady vortex patches, \cite{BG}  about steady flows related to solutions of the mean field equation, and \cite{W1} about steady flows related to least energy solutions of the Lane-Emden equation.  However, the methods used therein are more or less based on some particularities of the  flows,  and  seems not helpful with our problem.

A remaining open question is the stability of steady Euler flows with concentrated vorticity distributed in more than two small regions. In this case, there are at least two blobs of vorticity with the same sign, and consequently the separation property for the stream function can not  possibly hold.

This paper is organized as follows. In Section 2, we present several useful lemmas that will be frequently used in subsequent sections. In Section 3, we prove a general stability criterion. In Sections 4 and 5, we give the proof of Theorem \ref{el2k2} for $k=1$ and $k=2$ respectively. In Section 6, we give the proof of Theorem \ref{thmex} for completeness. In Section 7, we  briefly  discuss  the stability of concentrated vortex flows with prescribed profile functions.

 \section{Preliminaries}

The aim of this section is to present several useful lemmas for later use.
To make the statement concise, we only give the version that is enough for our purpose.

\begin{lemma}\label{lem201}
Let $1\leq p< +\infty,$ $\Omega\subset\mathbb R^2$ be a Lebesgue measurable set such that $\mathcal L(\Omega)<+\infty$, $f\in L^p(\Omega),$ and $\bar {\mathcal R}_f$ be the weak closure of $\mathcal R_f$ in $L^p(\Omega).$ Then $\bar {\mathcal R}_f$ convex.
\end{lemma}

\begin{proof}
See Theorem 6 in \cite{B1}.
\end{proof}

\begin{lemma}\label{lem2011}
Let $V,W\subset\mathbb R^2$ be two Lebesgue measurable sets such that $\mathcal L(V)= \mathcal L(W)<+\infty$. Then for any $w\in L^1(W),$ there exists some $v\in L^1(V)$ such that
\[\mathcal L(\{x\in W\mid w(x)>s\})=\mathcal L(\{x\in V\mid v(x)>s\}),\quad\forall \,s\in\mathbb R.\]
\end{lemma}

\begin{proof}
This lemma is an easy consequence of Lemma 4, (ii) in \cite{B1}.
\end{proof}

\begin{lemma}\label{lem202}
Let $1\leq p\leq +\infty,$ $\Omega\subset\mathbb R^2$ be a Lebesgue measurable set such that $\mathcal L(\Omega)<+\infty$. Let $f\in L^p(\Omega), g\in L^{p'}(\Omega)$ be fixed, where $p'$ is the H\"older conjugate exponent of $p$.  Then  there exists $\tilde f\in \mathcal R_f$, such that
\[\int_\Omega \tilde f gdx=\sup_{u\in \mathcal R_f, v\in \mathcal R_{g}}\int_\Omega uvdx.\]
\end{lemma}

\begin{proof}
See Theorem 1 and Theorem 4 in \cite{B1}.
\end{proof}

\begin{lemma}\label{yyds}
Let  $p\in[1,+\infty)$ be fixed, $\Omega\subset\mathbb R^2$ be a smooth bounded domain, $f_0\in L^p(\Omega)$ be nonnegative, and $g\in H^2(\Omega)$. Suppose $\tilde f\in\bar{\mathcal R}_{f_0}$ satisfies
\begin{itemize}
\item[(i)]$\int_\Omega \tilde fg dx\geq \int_\Omega fg dx$ for all $f\in\bar{\mathcal R}_{f_0};$
\item[(ii)] $-\Delta g\geq \tilde f$ a.e. $x\in\Omega$.
\end{itemize}
Then $\tilde f\in\mathcal R_{f_0}$. Moreover, there exists some nondecreasing function $\phi:\mathbb R\to[-\infty,+\infty]$ such that $\tilde f(x)=\phi(g(x))$ a.e. $x\in\Omega.$

\end{lemma}

\begin{proof}
See Lemma 2.15 in \cite{B2}.
\end{proof}

 \begin{lemma}\label{etans}
Let  $\Omega\subset\mathbb R^2$ be a smooth bounded domain, $\omega\in L^\infty((0,+\infty)\times \Omega)$, $\zeta_0\in L^\infty( \Omega)$. Then there exists a unique $\zeta\in L^\infty((0,+\infty)\times  \Omega)$ satisfying the following conditions:
\begin{itemize}
\item[(i)] $\zeta$ solves $\partial_t\zeta +\nabla^\perp \mathcal G\omega\cdot\nabla\zeta=0$ in the sense of distributions, i.e.,
\begin{equation}\label{ydmf}
\int_0^\infty\int_{ \Omega} \zeta\partial_t\phi+\zeta\nabla^\perp\mathcal G\omega\cdot\nabla\phi dxdt=0,\quad \forall\,\phi\in C_c^\infty((0,+\infty)\times  \Omega);
\end{equation}
\item [(ii)]$\zeta\in C([0,+\infty); L^p( \Omega))$  for any $p\in[1,+\infty)$;\item[(iii)] $\zeta(0,\cdot)=\zeta_0;$
\item[(iv)] $\zeta(t,\cdot)\in\mathcal R_{\zeta_0}$ for any $t\in[0,+\infty)$.
\end{itemize}
\end{lemma}

\begin{proof}
See Lemma 11 and Lemma 12 in \cite{B5}.
\end{proof}

The following two rearrangement inequalities can be found in Lieb and Loss's book \cite{LL}.
\begin{lemma}\label{rri1}
Let $f,g$ be nonnegative Lebesgue measurable functions on $\mathbb R^2$, and $f^*,g^*$ be their symmetric-decreasing rearrangements. Then
\[\int_{\mathbb R^2}fgdx\leq \int_{\mathbb R^2}f^*g^*dx.\]
\end{lemma}
\begin{proof}
See \S 3.4 in \cite{LL}.
\end{proof}

\begin{lemma}\label{rri2}
Let $f,g,h$ be nonnegative Lebesgue measurable functions on $\mathbb R^2$, and $f^*,g^*,h^*$ be their symmetric-decreasing rearrangements. Then
\[\int_{\mathbb R^2}\int_{\mathbb R^2}f(x)g(x-y)h(y)dxdy\leq \int_{\mathbb R^2}\int_{\mathbb R^2}f^*(x)g^*(x-y)h^*(y)dxdy.\]
\end{lemma}
\begin{proof}
See \S 3.7 in \cite{LL}.
\end{proof}

\section{A general stability criterion}

Throughout this section let $p\in(1,+\infty)$ be fixed, and $\mathcal R$ be the rearrangement class of some fixed function  $f_0\in L^\infty(D)$. Denote $\bar {\mathcal R}$ the weak closure of $\mathcal R$ in $L^p(D)$. For any nonempty set $A\subset L^p(D),$ denote $A_\delta$ the $\delta$-neighborhood of $A$ in $L^p(D)$, that is,
\[A_\delta=\left\{w\in L^p(D)\mid  {\rm dist}(w,A)<\delta\right\},\]
where \[{\rm dist}(w,A)=\inf_{v\in A}\|w-v\|_{L^p(D)}.\]
Obviously $A_\delta$ is an open set in $L^p(D)$.

The purpose of this section is to prove the following stability criterion.
\begin{theorem}\label{thm551}
 Suppose $\mathcal M\subset \mathcal R$ is nonempty and satisfies
\begin{itemize}
\item [(i)] $\mathcal M$ is compact in $L^p(D)$;
\item[(ii)]$\mathcal M$ is an isolated set of local maximizers of $E$ over $\mathcal R,$ that is, there exists some $\delta>0$ such that
\begin{equation}\label{dct}
\mathcal M=\{w\in \mathcal M_\delta\cap \mathcal R\mid E(w)=E_0\}, \mbox{ where }E_0=\sup_{w\in \mathcal M_\delta\cap \mathcal R}E(w).
\end{equation}
\end{itemize}
Then any $\omega\in\mathcal M $ is a  steady solution  to the  vorticity equation in the sense of Definition \ref{wsv1}, and $ \mathcal M $ is stable in the $L^p$ norm of the vorticity with respect to initial perturbations in $L^\infty(D)$.
\end{theorem}

\begin{remark}
The  assumption (ii) in Theorem \ref{thm551} holds if and only the following two items are satisfied
\begin{itemize}
\item[(1)]$E$ is a constant  on $\mathcal M;$
\item[(2)] for any $v\in(\mathcal M)_\delta\cap\mathcal R$ satisfying $E(v)\geq E|_{\mathcal M}$, it holds that $v\in \mathcal M.$
\end{itemize}
\end{remark}

\begin{remark}
If $\mathcal M$ is a singleton, then Theorem \ref{thm551} is exactly Theorem 5 in \cite{B5}.
\end{remark}

\begin{remark}\label{lrem}
Whether the compactness assumption (i) in Theorem \ref{thm551} can be removed is unknown. There are two special cases in which compactness holds automatically:
\begin{itemize}
\item[(1)] $\mathcal M$ contains only a finite number of elements;
\item [(2)] $E_0=\sup_{w\in \mathcal R}E(w)$.
\end{itemize}
The proof for the first case is clear. For the second case, we give a brief explanation  as follows. First, from the weak sequential continuity of $E$ in $L^p(D)$, we have
 \begin{equation}
E_0=\sup_{w\in \mathcal R}E(w)=\sup_{w\in \bar{\mathcal R}}E(w).
\end{equation}
Fix a sequence $\{w_n\}_{n=1}^{+\infty}\subset \mathcal M$. Then there exist a subsequence $\{w_{n_j}\}_{j=1}^{+\infty}$ and some $\eta\in\bar{\mathcal R}$ such that $w_{n_j}$ converges to $\eta$ weakly in $L^p(D)$ as $j\to+\infty$. The weakly sequential continuity of $E$ in $L^p(D)$ yields  $E(\eta)=E_0$. Since $\bar{\mathcal R}$ is convex by Lemma \ref{lem201},  for any $v\in\bar{\mathcal R}$ and $s\in[0,1]$ we have $sv+(1-s)\eta\in \bar{\mathcal R}.$ Hence
\[\frac{d}{ds}E(sv+(1-s)\eta)\bigg|_{s=0^+}\leq 0,\]
which gives
\begin{equation}\label{sbd}
\int_Dv\mathcal G\eta dx\leq \int_D\eta \mathcal G \eta dx.
\end{equation}
Since in \eqref{sbd}  the choice of $v$ in $\bar{\mathcal R}$ is arbitrary, Lemma \ref{yyds} yields $\eta\in\mathcal R.$ By the uniform convexity of $L^p(D)$, we see that $w_{n_j}$ converges to $\eta$ strongly in $L^p(D)$ as $j\to+\infty$, and consequently $\eta\in\mathcal M_\delta\cap\mathcal R.$ Taking into account \eqref{dct} we get $\eta\in\mathcal M.$

\end{remark}

To prove Theorem \ref{thm551}, we need several auxiliary lemmas.

\begin{lemma}\label{lem552}
The strong and weak topologies of $L^p(D)$ restricted on $\mathcal R$ are the same.
\end{lemma}
\begin{proof}
It suffices to show that any strongly closed set $F\subset \mathcal R$ must be weakly closed (both in the sense of subspace topology of $\mathcal R$). Since $\mathcal R$ is bounded in $L^p(D)$ and $p\in(1,+\infty), $ the weak topology restricted on $\mathcal R$ is metrizable. Hence it suffices to prove that $F$ is weakly sequentially closed on $\mathcal R$. Suppose  $\{w_n\}_{n=1}^{+\infty}\subset F$  and $w_n$ converges weakly to some $\eta\in\mathcal R.$ Then $\|\eta\|_{L^p(D)}=\lim_{n\to+\infty}\|w_n\|_{L^p(D)}$. By the fact that $L^p(D)$ is uniformly convex, we immediately see that  $w_n$ converges strongly to  $\eta,$ which yields $\eta\in F$.
\end{proof}

\begin{remark}
From the proof of Lemma \ref{lem552}, it is easy to check that for any $\mathcal N\subset \mathcal R$ the following three items are equivalent:
\begin{itemize}
\item[(1)] $\mathcal N$ is compact;
\item[(2)]  $\mathcal N$ is weakly closed;
\item[(3)] $\mathcal N$ is weakly sequentially closed.
\end{itemize}
\end{remark}

\begin{lemma}\label{lem553}
There exists a weakly open set $U$ such that $U\cap\mathcal R=\mathcal M_\delta\cap\mathcal R.$
\end{lemma}

\begin{proof}

Observe that $\mathcal M_\delta\cap \mathcal R$ is strongly open on $\mathcal R$ (in the subspace topology). By Lemma \ref{lem552},  $\mathcal M_\delta\cap \mathcal R$ is also  weakly open on $\mathcal R$. Hence the existence of the desired $U$ follows immediately.
\end{proof}

\begin{lemma}\label{lem558}
There exists $\tau\in(0,\delta)$ such that $\mathcal M_\tau\subset U.$
\end{lemma}
\begin{proof}
Suppose by contradiction that there exists a sequence $\{w_n\}_{n=1}^{+\infty}\subset U^c$ such that
\[{\rm dist}(w_n,\mathcal M)<\frac{1}{n}.\]
Here $U^c$ is the complement of $U$.
Since $U$ is weakly open, $U$ must be strongly open, which means $U^c$ is strongly closed.
Choose a sequence  $\{v_n\}_{n=1}^{+\infty}\subset \mathcal M$ such that
\begin{equation}\label{sscf}
\|w_n-v_n\|_{L^p(D)}<\frac{1}{n}.
\end{equation}
Since $\mathcal M$ is compact, there exist a subsequence $\{v_{n_j}\}_{j=1}^{+\infty}$  and some $v\in\mathcal M$ such that
\[\lim_{j\to+\infty}\|v_{n_j}-v\|_{L^p(D)}=0,\]
which together with \eqref{sscf} yields
 \[\lim_{j\to+\infty}\|w_{n_j}-v\|_{L^p(D)}=0.\] Taking into account the fact that $U^c$ is strongly closed, we get $v\in U^c$. This contradicts the fact that $v\in\mathcal M\subset U.$

\end{proof}

\begin{lemma}\label{lem559}
Let $\tau$ be determined by Lemma \ref{lem558}, then
\begin{equation}\label{678}
E_0=\sup_{w\in\mathcal M_\tau\cap\bar{\mathcal R}}E(w).
\end{equation}
\end{lemma}
\begin{proof}
First by \eqref{dct} we have $E(w)=E_0$ for any $w\in\mathcal M,$ thus
\[E_0\leq \sup_{w\in \mathcal M_\tau\cap\bar{\mathcal R}}E(w). \]
So it suffices to prove the inverse inequality, i.e.,
\begin{equation}\label{566}
E(w)\leq E_0,\,\,\forall\,w\in \mathcal M_\tau\cap\bar{\mathcal R}.
\end{equation}
Fix $w\in \mathcal M_\tau\cap\bar{\mathcal R}.$ Since $\bar {\mathcal R}$ is the weak closure of $\mathcal R$, there exists a sequence $\{w_n\}_{n=1}^{+\infty}\subset \mathcal R$ such that $w_n$ converges weakly to $w$ as $n\to+\infty.$ On the other hand, since $w\in\mathcal M_\tau\subset U$ (by Lemma \ref{lem558}) and $U$ is weakly open, we obtain $w_n\in U,\,\,\forall\,n\geq N_0$ for some sufficiently large $N_0.$ Therefore
\[w_n\in U\cap\mathcal R\subset \mathcal M_\delta\cap \mathcal R,\,\,\forall\,n\geq N_0.\]
Taking into account the definition of $E_0$ (see \eqref{dct}), we have $E(w_n)\leq E_0$ for all $n\geq N_0.$ Hence
\[E(w)=\lim_{n\to+\infty}E(w_n)\leq E_0.\]

\end{proof}

\begin{lemma}\label{lem560}
Let $\tau$ be determined by Lemma \ref{lem558}, then
\begin{equation}\label{679}
\mathcal M=\{w\in\mathcal M_\tau\cap\bar{\mathcal R}\mid E(w)=E_0\}.
\end{equation}
\end{lemma}
\begin{proof}
It suffices to prove
\[\{w\in\mathcal M_\tau\cap\bar{\mathcal R}\mid E(w)=E_0\}\subset \mathcal M.\]
To this end, fix $w\in\mathcal M_\tau\cap\bar{\mathcal R}$ such that $E(w)=E_0$. By \eqref{dct}, we need only to show that $w\in\mathcal R.$
The argument is similar to that in Remark \ref{lrem}. Since $\bar{\mathcal R}$ is a convex set by Lemma \ref{lem201}, we have
\[sv+(1-s)w\in\bar{\mathcal R},\quad\forall\, v\in\bar{\mathcal R},\,\,s\in[0,1].\]
 On the other hand, it is  clear that
\[sv+(1-s)w\in \mathcal M_\tau,\]
 and thus
 \[sv+(1-s)w\in\mathcal M_\tau\cap\bar{\mathcal R},\] provided that $s\geq 0$ is sufficiently small. In view of Lemma \ref{lem559},    we obtain
\[\frac{d}{ds}E(sv+(1-s)w)\bigg|_{s=0^+}\leq 0,\]
which yields
\begin{equation}\label{yds}
\int_Dv\mathcal Gw dx\leq \int_Dw\mathcal G wdx.
\end{equation}
Since  \eqref{yds}  holds for any  $v\in\bar{\mathcal R}$, we can apply Lemma \ref{yyds} to obtain $w\in\mathcal R.$ This finishes the proof.

\end{proof}

\begin{lemma}\label{lem561}
Let $\tau$ be determined by Lemma \ref{lem558}. Then for any $\{w_n\}_{n=1}^{+\infty}\subset \mathcal M_{\frac{\tau}{2}}\cap \bar{\mathcal R}$ satisfying $\lim_{n\to+\infty}E(w_n)=E_{0},$  there exist a
subsequence $\{w_{n_j}\}_{j=1}^{+\infty}$ and some $\eta\in\mathcal M$ such that $w_{n_j}$ converges to $\eta$ strongly in $L^p(D)$ as $j\to+\infty.$
\end{lemma}

\begin{proof}

Since $\bar{\mathcal R}$ is weakly sequentially closed in $L^p(D)$, there exist a subsequence $\{w_{n_j}\}_{j=1}^{+\infty}$ and some $\eta\in\bar{\mathcal R}$ such that $w_{n_j}$ converges to $\eta$ weakly in $L^p(D)$ as $j\to+\infty.$ To finish the proof, it suffices to show that $\eta\in\mathcal M$. In fact, if $\eta\in\mathcal M\subset\mathcal R,$ then  $\|w_{n_{j}}\|_{L^{p}(D)}\leq \|\eta\|_{L^{p}(D)}$ for any $j$. In combination with the  following  fact (due to the weak lower semicontinuity of the $L^{p}$ norm):
\[\|\eta\|_{L^{p}(D)}\leq \liminf_{j\to+\infty}\|w_{n_{j}}\|_{L^{p}(D)},\]
we obtain
\[\lim_{j\to+\infty}\|w_{n_{j}}\|_{L^{p}(D)}=\|\eta\|_{L^{p}(D)}.\]
Therefore  by uniform convexity  $w_{n_{j}}$ converges strongly to $\eta$ in $L^{p}(D)$ as $j\to+\infty$ as required.

Now we prove $\eta\in\mathcal M$. First it is clear from the weak sequential continuity of $E$ in $L^p(D)$ that
\begin{equation}\label{clrr}
E(\eta)=E_0.
\end{equation}
Second, we can show that $\eta\in\mathcal M_\tau\cap\bar{\mathcal R}.$ In fact, since $\{w_{n_j}\}_{j=1}^{+\infty}\subset\mathcal M_{\frac{\tau}{2}}$, there exists a sequence $\{v_j\}_{j=1}^{+\infty}\subset\mathcal M$ such that
\[\|w_{n_j}-v_j\|_{L^p(D)}<\frac{\tau}{2}.\]
Since $\mathcal M$ is compact, there exist a subsequence $\{v_{j_k}\}_{k=1}^{+\infty}$ and some $v\in\mathcal M$ such that $v_{j_k}$ converges to $v$ strongly in $L^p(D) $ as $k\to+\infty.$ Hence
\[\|\eta-v\|_{L^p(D)}\leq \liminf_{k\to+\infty}\|w_{n_{j_k}}-v_{j_k}\|_{L^p(D)}\leq \frac{\tau}{2},\]
which implies $\eta\in\mathcal M_{\tau}.$ Therefore
 \begin{equation}\label{hhg}
\eta\in\mathcal M_\tau\cap\bar{\mathcal R}.
\end{equation}

Now \eqref{clrr}, \eqref{hhg} and Lemma \ref{lem560} together yield $\eta\in \mathcal M$.
\end{proof}

Now we are ready to prove Theorem \ref{thm551}.

\begin{proof}[Proof of Theorem \ref{thm551}]
First we show that any $\omega\in\mathcal M$ is a steady weak solution to the vorticity equation. For any $\zeta\in C_c^\infty(D)$,  define a family of smooth area-preserving transformations $\Phi_t: D\to D$  as follows
\[\begin{cases}
\frac{d}{dt}\Phi_{t}(x)=\nabla^\perp \zeta(x),&t\in\mathbb R,\\
\Phi_0(x)=x,&x\in D.
\end{cases}
\]
It is clear that $\omega(\Phi_{-t}(\cdot))\in \mathcal R_{\omega}$ for all $t\in\mathbb R.$ Taking into account (ii) in Theorem \ref{thm551}, we see  that $E(\omega(\Phi_{-t}(\cdot)))$ attains its local maximum value at $t=0$, which implies
\begin{equation}\label{gogoa}
\frac{d}{dt}E(\omega(\Phi_{-t}(\cdot))\bigg|_{t=0}=0.
\end{equation}
By a simple calculation we get from \eqref{gogoa} that
\[\int_D\omega\nabla^\perp\mathcal G\omega\cdot\nabla\zeta dx=0,\]
which means that $\omega$ is a steady   solution to the vorticity equation in the sense of Definition \ref{wsv1}.

Now we prove the stability of $\mathcal M$. Suppose by contradiction that $\mathcal M$ is not stable in the $L^p$ norm of the vorticity with initial perturbations in $L^\infty(D)$. Then there exist some $\epsilon_0>0$, a sequence $\{\omega_0^n\}_{n=1}^{+\infty}\subset L^\infty(D)$, and a sequence $\{t_n\}_{n=1}^{+\infty}\subset \mathbb R_+$, such that
\begin{equation}\label{cbb1}
{\rm dist}(\omega^n_0,\mathcal M)<\frac{1}{n},
\end{equation}
and
\[{\rm dist}(\omega^n_{t_n},\mathcal M)\geq\epsilon_0,\]
where $\omega^n_{t_n}$ is the unique weak solution to the vorticity equation at  time $t_n$ with initial vorticity $\omega^n_0.$ Without loss of generality, we assume that
\begin{equation}\label{cbb2}
0<\epsilon_0<\frac{\tau}{3},
\end{equation}
where $\tau$ is the positive number determined in Lemma \ref{lem558}. Moreover, since any weak solution to the vorticity is continuous with respect to the time variable in $L^p(D)$ (see Yudovich's Theorem in Section 1), without loss of generality we can assume that
\begin{equation}\label{cbb3}
{\rm dist}(\omega^n_{t_n},\mathcal M)=\epsilon_0 \,\,\mbox{ for each }n.
\end{equation}
From \eqref{cbb2} and \eqref{cbb3} we immediately get
\begin{equation}\label{cbb100}
\omega^n_{t_n}\in\mathcal M_{\frac{\tau}{3}}\,\,\mbox{ for each }n.
\end{equation}

By \eqref{cbb1}, we can choose a sequence $\{w_n\}_{n=1}^{+\infty}\subset\mathcal M$ such that
\begin{equation}\label{cbb4}
\|\omega^n_0-w_n\|_{L^p(D)}<\frac{1}{n}.
\end{equation}
Since $\mathcal M$ is compact in $L^p(D)$, there exist a subsequence, still denoted by $\{w_n\}_{n=1}^{+\infty},$ and some $w_0\in\mathcal M$, such that $w_n$ converges strongly to $w_0$ as $n\to+\infty.$ Combining \eqref{cbb4} we get
\begin{equation}\label{cbb5}
\lim_{n\to+\infty}\|\omega^n_0-w_0\|_{L^p(D)}=0.
\end{equation}
Taking into account energy conservation we get
\begin{equation}\label{cbb6}
\lim_{n\to+\infty}E(\omega^n_{t_n})=\lim_{n\to+\infty}E(\omega^n_{0})=E(w_0)=E_0.
\end{equation}

Now for each fixed $n$ let $w^n$ be the weak solution to the following linear transport equation with initial value $w_0$ (in the sense of Lemma \ref{etans})
\[\partial w^n+\nabla^\perp\mathcal G\omega^n\cdot\nabla w^n=0.\]
Then for each $n$ it holds that
\begin{equation}\label{cbb103}
w^n_{t_n}\in\mathcal R_{w_0}=\mathcal R,
\end{equation}
\begin{equation}\label{cbb104}
w^n_{t_n}-\omega^n_{t_n}\in\mathcal R_{w_0-\omega^n_0}.
\end{equation}
By \eqref{cbb104} we get
\begin{equation}\label{cbb7}
\|w^n_{t_n}-\omega^n_{t_n}\|_{L^p(D)}=\|w_0-\omega^n_0\|_{L^p(D)}\to 0 \mbox{ as }n\to+\infty.
\end{equation}
This together with \eqref{cbb100} yields
\begin{equation}\label{cbb101}
w^n_{t_n}\in\mathcal M_{\frac{\tau}{2}} \,\,\mbox{ if $n$ is sufficiently large}.
\end{equation}
From \eqref{cbb103} and \eqref{cbb101} we obtain
\begin{equation}\label{cbb109}
w^n_{t_n}\in\mathcal R\cap \mathcal M_{\frac{\tau}{2}} \,\,\mbox{ if $n$ is sufficiently large}.
\end{equation}
On the other hand, by \eqref{cbb6} and \eqref{cbb7} we have
\begin{equation}\label{cbb120}
\lim_{n\to+\infty}E(w^n_{t_n})=E_0.
\end{equation}
To summarize, we have obtained a sequence $\{w^n_{t_n}\}_{n=1}^{+\infty}$ satisfying \eqref{cbb109} and \eqref{cbb120}.  By Lemma \ref{lem561}, there exist a subsequence, still denoted by $\{w^n_{t_n}\}_{n=1}^{+\infty}$, and some $\eta\in\mathcal M$ such that $w^n_{t_n}$ converges strongly to $\eta$ in $L^p(D)$ as $n\to+\infty$. Taking into account \eqref{cbb7} we deduce  that $\omega^n_{t_n}$ converges strongly to $\eta$  in $L^p(D)$ as $n\to+\infty$, which implies
\begin{equation}\label{cbb122}
\lim_{n\to+\infty}{\rm dist}(\omega^n_{t_n},\mathcal M)=0.
\end{equation}
This is an obvious contradiction to \eqref{cbb3}. The proof is finished.

\end{proof}

\section{Proof of Theorem \ref{el2k2}: $k=1$ }
In this section we give the proof of Theorem \ref{el2k2} for $k=1$.

Throughout this section, let $1<p<+\infty,  M>0, \kappa>0$ be fixed, and $\bar x\in D$ be an isolated local minimum point of the Robin function $H$ (defined by \eqref{ropz89}).
Fix a small positive number $\bar r$ such that
\begin{itemize}
\item[(i)] $\overline{B_{\bar r}(\bar x)}\subset D$;
\item[(ii)] $\bar x$ is the unique minimum point of $H$ in $\overline{B_{\bar r}(\bar x)}$.
\end{itemize}

Let $\Xi$ be defined by \eqref{jjk}, and  $\Pi:(0,\bar r)\to\Xi$ be a map satisfying
 \begin{equation}\label{h134}
 \begin{split}
 & {\mbox{supp($\Pi_\varepsilon$)}}\subset \overline{B_\varepsilon(\mathbf 0)}
\,\,\mbox{ for any } \varepsilon\in(0,\bar r),\\
& \lim_{\varepsilon\to0}\int_{\mathbb R^2}\Pi_\varepsilon(x)dx=\kappa,\\
& \|\Pi_\varepsilon\|_{L^\infty(\mathbb R^2)}\leq M\varepsilon^{-2}.
\end{split}
\end{equation}
Below for convenience we denote
\begin{equation}\label{del1}
\kappa_\varepsilon=\int_{\mathbb R^2}\Pi_\varepsilon(x)dx.
\end{equation}
Define
 \[\mathcal R_{\varepsilon}=\left\{w\in L^\infty(D) \mid w \mbox{ is a rearrangement of } \Pi_\varepsilon\right\}.\]
The $\mathcal K_{\varepsilon}$ defined by \eqref{kvar} becomes
\[\mathcal K_{\varepsilon}=\left\{w\in \mathcal R_{\varepsilon}\mid {\rm supp}(w)\subset \overline{B_{\bar r}(\bar x)}\right\}.\]
Denote $\mathcal M_{\varepsilon}$ the set of maximizers of $E$ over $\mathcal K_{\varepsilon}.$

Now Theorem \ref{thmex} can be restated as follows.

 \begin{theorem}\label{thmex1}
 For any $\varepsilon\in(0,\bar r)$,  $\mathcal M_{\varepsilon}$ is nonempty,  and the following assertions hold.
\begin{itemize}
\item[(i)] There exists some $\varepsilon_1>0$, such that for any $\varepsilon\in(0,\varepsilon_1)$,  $ \mathcal M_{\varepsilon}$ is a set of  steady solutions to the  vorticity  equation.
\item [(ii)]  For any $\omega\in\mathcal M_\varepsilon,$ $\mbox{ supp}(\omega)$ ``shrinks" to $\bar x$  uniformly as $\varepsilon\to0$. More precisely,
for any $\delta\in(0,\bar r),$ there exists some $\varepsilon_2>0,$  such that for any $\varepsilon\in(0,\varepsilon_2)$, it holds that
\[\mbox{\rm supp}(\omega)\subset \overline{B_{\delta}(\bar x)},\quad\forall\,\omega \in\mathcal M_{\varepsilon}.\]
\end{itemize}
 \end{theorem}

 By Theorem \ref{thm551}, in order to prove the stability of $\mathcal M_{\varepsilon},$ it suffices to show that $\mathcal M_{\varepsilon}$ is compact and isolated in the sense of Theorem \ref{thm551}.

First we prove compactness, which relies on Lemma \ref{yyds}.
\begin{proposition}[Compactness]\label{comlem}
For any $\varepsilon\in(0,\bar r)$, $\mathcal M_{\varepsilon}$ is compact in $L^p(D)$.
\end{proposition}
\begin{proof}
Fix $\varepsilon\in(0,\bar r)$. Let  $\{\omega_n\}_{n=1}^{+\infty}$ be an arbitrary sequence in $\mathcal M_{\varepsilon}.$ We will show that it contains a subsequence that converges to some element of $\mathcal M_{\varepsilon}$ in $L^p(D)$.

Obviously $\{\omega_n\}_{n=1}^{+\infty}$ is bounded in $L^p(D)$, thus there exist a subsequence, still denoted by $\{\omega_n\}_{n=1}^{+\infty}$, and some $\eta\in\bar{\mathcal K}_{\varepsilon}$, such that $\omega_n$ converges weakly to $\eta$ in $L^p(D)$ as $n\to+\infty.$ Here $\bar{\mathcal K}_{\varepsilon}$ denotes the weak closure of ${\mathcal K}_{\varepsilon}$ in $L^p(D)$, which is also equal to the weak closure of ${\mathcal K}_{\varepsilon}$ in $L^p(B_{\bar r}(\bar x))$.
It is clear that
\begin{equation}\label{op1}
E(\eta)=\lim_{n\to+\infty}E(\omega_n)=\sup_{w\in {\mathcal K}_{\varepsilon}}E(w)=\sup_{w\in \bar{\mathcal K}_{\varepsilon}}E(w).
\end{equation}
To finish the proof, it is sufficient to show that $\eta\in\mathcal K_{\varepsilon}$ (obviously this implies $\eta\in\mathcal M_{\varepsilon}$ and $\omega_n$ converges strongly to $\eta$ as $n\to+\infty$).
Since $\bar{\mathcal K}_{\varepsilon}$ is convex by Lemma \ref{lem201},
 for any $v\in\bar{\mathcal K}_{\varepsilon}$ and $s\in[0,1]$ we have $sv+(1-s)\eta\in\bar{\mathcal K}_{\varepsilon}$. Hence
 \[\frac{d}{ds}E(sv+(1-s)\eta)\bigg|_{s=0^+}\leq 0,\]
which yields
\[\int_{B_{\bar r}(\bar x)}v\mathcal G\eta dx\leq\int_{B_{\bar r}(\bar x)}\eta\mathcal G\eta dx. \]
Since $v\in\bar{\mathcal K}_{\varepsilon}$ is arbitrary, we can take $\Omega=B_{\bar r}(\bar x)$ and $\mathcal R_{f_0}=\mathcal K_{\varepsilon}$ in Lemma \ref{yyds} to get $\eta\in \mathcal K_{\varepsilon}.$

\end{proof}

Now we turn to the proof of isolatedness, which is a little complicated and  only holds when $\varepsilon$ is sufficiently small. To begin with, we prove the following separation property for the stream function.

\begin{lemma}[Separation property]\label{sepalem}
 Define
\[r_{\varepsilon}=\inf\{r>0\mid \mbox{for any }\omega\in\mathcal M_{\varepsilon}, \,\,\omega=0\mbox{ a.e. in } D\setminus B_{r}(\bar x)\}\]
Then there exists $\varepsilon_3>0$,   such that  for any $\varepsilon\in(0,\varepsilon_3)$, it holds  that
\begin{equation}\label{sepa1}
\sup_{D\setminus B_{\bar r}(\bar x)}\mathcal G\omega<\inf_{{B_{r_\varepsilon}(\bar x)}}\mathcal G\omega,\quad\forall\,\omega\in \mathcal M_{\varepsilon}.
\end{equation}

\end{lemma}

\begin{proof}

Obviously supp$(\omega)\subset\overline{B_{r_\varepsilon}(\bar x)}$ for any $\omega\in\mathcal M_\varepsilon$. This implies $r_{\varepsilon}>0$ for any $\varepsilon\in(0,\bar r)$. Moreover, by (i) in Theorem \ref{thmex1} we see that $r_{\varepsilon}\to 0$ as $\varepsilon\to 0$.

Fix $\omega\in\mathcal M_{\varepsilon}$.  For any $x\in D\setminus B_{\bar r}(\bar x),$ we estimate $\mathcal G\omega(x)$ as follows
\begin{equation}\label{xsw1}
\begin{split}
\mathcal G\omega(x)&=-\frac{1}{2\pi}\int_D\ln|x-y|\omega(y)dy-\int_Dh(x,y)\omega(y)dy \\
&=-\frac{1}{2\pi}\int_{B_{r_{\varepsilon}}(\bar x)}\ln|x-y|\omega(y)dy-\int_Dh(x,y)\omega(y)dy\\
&\leq -\frac{1}{2\pi}\ln(\bar r-r_{\varepsilon})\int_{D}\omega(y)dy+ \|h\|_{L^\infty(B_{\bar r}(\bar x)\times B_{\bar r}(\bar x))}\int_D\omega(y)dy\\
&=-\frac{\kappa_\varepsilon}{2\pi}\ln(\bar r-r_{\varepsilon})+ \|h\|_{L^\infty(B_{\bar r}(\bar x)\times B_{\bar r}(\bar x))}\kappa_\varepsilon\\
&\to-\frac{\kappa}{2\pi}\ln\bar r+\|h\|_{L^\infty(B_{\bar r}(\bar x)\times B_{\bar r}(\bar x))}\kappa
\end{split}
\end{equation}
as $\varepsilon\to0.$ Here we used the fact that  $\omega=0$ a.e. in $ D\setminus B_{r_{\varepsilon}}(\bar x)$. This can be easily verified from the definition of $r_{\varepsilon}$.

On the other hand, for any  $x\in  {B_{ r_{\varepsilon}}(\bar x)},$ we have
\begin{equation}\label{xsw2}
\begin{split}
\mathcal G\omega(x)&=-\frac{1}{2\pi}\int_D\ln|x-y|\omega(y)dy-\int_Dh(x,y)\omega(y)dy\\
&=-\frac{1}{2\pi}\int_{B_{r_{\varepsilon}}(\bar x)}\ln|x-y|\omega(y)dy-\int_Dh(x,y)\omega(y)dy\\
&\geq -\frac{1}{2\pi}\ln(2r_{\varepsilon})\int_{D}\omega(y)dy- \|h\|_{L^\infty(B_{\bar r}(\bar x)\times B_{\bar r}(\bar x))}\int_D\omega(y)dy\\
&= -\frac{\kappa_\varepsilon}{2\pi}\ln(2r_{\varepsilon})- \|h\|_{L^\infty(B_{\bar r}(\bar x)\times B_{\bar r}(\bar x))}\kappa_\varepsilon\\
&\to+\infty
\end{split}
\end{equation}
as $\varepsilon\to0.$

Combining \eqref{xsw1} and \eqref{xsw2}, we deduce that
\begin{equation*}
\sup_{D\setminus B_{\bar r}(\bar x)}\mathcal G\omega<\inf_{{B_{r_\varepsilon}(\bar x)}}\mathcal G\omega,\quad\forall\,\omega\in \mathcal M_{\varepsilon}
\end{equation*}
if $\varepsilon$ is sufficiently small.

\end{proof}

The following lemma will also be used in obtaining isolatedness.
\begin{lemma}\label{fjin}
Let $U\subset \mathbb R^2$ be a bounded domain, $B_1,B_2$ be open balls such that $B_1\subset\subset B_2\subset\subset U$, and $\psi\in C(\bar U)$ satisfies
\begin{equation}\label{fjin1}
\inf_{B_1}\psi>\sup_{U\setminus B_2}\psi.
\end{equation}
Let $\mathcal R$ be the rearrangement class of some nonnegative function $f_0\in L^1(U)$ satisfying
\begin{equation}\label{fjin4}
\mathcal L(\{x\in U\mid f_0(x)>0\})\leq \mathcal L(B_1).
\end{equation}
Suppose $u\in\mathcal R$ satisfies
\begin{equation}\label{fjin44}
\int_Uu\psi dx\geq \int_Uw\psi dx\quad\forall \,w\in\mathcal R,
\end{equation}
then supp$(u)\subset \bar B_2.$
\end{lemma}
\begin{proof}
Suppose by contradiction that  supp$(u)\not\subset \bar B_2.$ Define \[W:=\{x\in U\setminus B_2\mid u(x)>0\}.\]
Then it is obvious that
 \[\mathcal L(W)>0,\quad\int_Wu dx>0.\]
 Now we claim that
 \begin{equation}\label{2710}
 \mathcal L(\{x\in B_1\mid  u(x)=0\})\geq \mathcal L(W).
 \end{equation}
In fact,
\begin{align*}
&\mathcal L(\{x\in B_1\mid  u(x)=0\})\\
=&\mathcal L(B_1)-\mathcal L(\{x\in B_1\mid u(x)>0\})\\
\geq &\mathcal L(B_1)-\mathcal L(\{x\in B_2\mid u(x)>0\}) \\
=&\mathcal L(B_1)-\mathcal L(\{x\in U\mid u(x)>0\})+\mathcal L(\{x\in U\setminus B_2\mid u(x)>0\})\\
=&\mathcal L(B_1)-\mathcal L(\{x\in U\mid f_0(x)>0\})+\mathcal L(W)\\
\geq &\mathcal L(W).
\end{align*}
Note that in the last inequality we used \eqref{fjin4}.
By \eqref{2710}, there exists some Lebesgue measurable set $V\subset \{x\in B_1\mid  u(x)=0\}$ such that
\[\mathcal L(V)=\mathcal L(W).\]
Applying Lemma \ref{lem2011}, we can choose some Lebesgue measurable function $v$ defined on $V$, such that
$v$  is a rearrangement of $u\mathbf 1_W$. Extend $v$ such that $v(x)=0$ a.e. $x\in U\setminus V.$ Define
\[\hat u=u-u\mathbf 1_W+v.\]
Then it is clear that $\hat u\in\mathcal R.$ Now we calculate as follows
\begin{align*}
\int_Uu\psi dx-\int_U \hat u\psi dx&=\int_U(u\mathbf 1_W-v)\psi dx\\
&=\int_Wu\psi dx-\int_Vv\psi dx\\
&\leq \sup_{U\setminus B_2}\psi\int_Wu dx-\inf_{B_1}\psi\int_Vv dx\\
&=\left( \sup_{U\setminus B_2}\psi-\inf_{B_1}\psi\right)\int_Wu dx\\
&<0,
\end{align*}
which is a contradiction to \eqref{fjin44}.

\end{proof}

Now we are ready to prove the  isolatedness of $\mathcal M_\varepsilon$  when $\varepsilon$ is sufficiently small.

\begin{proposition}[Isolatedness]\label{isolated1}
Let $\varepsilon_3>0$ be determined by Lemma \ref{sepalem}. Then for any $\varepsilon\in(0,\varepsilon_3),$ there exists some $\delta>0$ such that
\begin{equation}\label{iso1}
\mathcal M_{\varepsilon}=\left\{w\in (\mathcal M_{\varepsilon})_\delta\cap \mathcal R_{\varepsilon}\,\,\bigg|\,\, E(w)=\sup_{w\in (\mathcal M_{\varepsilon})_\delta\cap \mathcal R_{\varepsilon}}E(w)\right\},
\end{equation}
where $(\mathcal M_{\varepsilon})_\delta$ is the $\delta$-neighborhood of $\mathcal M_{\varepsilon}$ in $L^p(D)$.
\end{proposition}

\begin{proof}
Let $\varepsilon\in(0,\varepsilon_3)$ be fixed. Then  \eqref{sepa1} holds  by Lemma \ref{sepalem}.

Denote \begin{equation}\label{iso678}
M_{\varepsilon}=\sup_{w\in\mathcal K_{\varepsilon}}E(w).
\end{equation}
Obviously for any $w\in\mathcal K_{\varepsilon}$, $w\in\mathcal M_{\varepsilon}$ if and only if $E(w)=M_{\varepsilon}$.

To prove \eqref{iso1}, it suffices to show that there exists some $\delta>0$ such that for any $w\in (\mathcal M_{\varepsilon})_\delta\cap \mathcal R_{\varepsilon}$ satisfying $E(w)\geq M_{\varepsilon}$, it holds that
$w\in \mathcal M_{\varepsilon}.$
 Suppose by contradiction that this is false. Then for any positive integer $n$, there exists some $w_n$ such that
  \begin{equation}\label{iso5}
  w_n\in  (\mathcal M_{\varepsilon})_{\frac{1}{n}}\cap \mathcal R_{\varepsilon},
  \end{equation}
    \begin{equation}\label{iso55}
 E(w_n)\geq M_{\varepsilon},
   \end{equation}
    \begin{equation}\label{iso555}
 w_n\notin \mathcal M_{\varepsilon}.
  \end{equation}
By \eqref{iso5}, we can choose some sequence $\{v_n\}_{n=1}^{+\infty}\subset \mathcal M_{\varepsilon}$ such that
 \begin{equation}\label{iso9}
 \|v_n-w_n\|_{L^p(D)}<\frac{1}{n}.
 \end{equation}
On the other hand, by the compactness of  $\mathcal M_{\varepsilon}$ we can choose a subsequence of $\{v_n\}_{n=1}^{+\infty},$ still denoted by $\{v_n\}_{n=1}^{+\infty},$ and some $\eta\in\mathcal M_{\varepsilon}$, such that
 \begin{equation}\label{iso99}
 \lim_{n\to+\infty}\|v_n-\eta\|_{L^p(D)}=0.
 \end{equation}
Combining \eqref{iso9} and \eqref{iso99} we obtain
 \begin{equation}\label{iso997}
\lim_{n\to+\infty}\|w_n-\eta\|_{L^p(D)}=0.
\end{equation}

By standard elliptic estimates, we get from \eqref{iso997} that
 \begin{equation}\label{iso999}
\lim_{n\to+\infty}\|\mathcal Gw_n-\mathcal G\eta\|_{L^\infty(D)}=0.
\end{equation}
On the other hand, since $\eta\in\mathcal M_{\varepsilon}$, we get from   \eqref{sepa1} that
\begin{equation}\label{iso8}
\sup_{{D}\setminus B_{\bar r}(\bar x)}\mathcal G\eta<\inf_{ {B_{r_\varepsilon}(\bar x)}}\mathcal G\eta.
\end{equation}
Now \eqref{iso999} and \eqref{iso8} together yield
\begin{equation}\label{iso88}
\sup_{D\setminus B_{\bar r}(\bar x)}\mathcal Gw_n<\inf_{ {B_{r_\varepsilon}(\bar x)}}\mathcal Gw_n
\end{equation}
if $n$ is sufficiently large.
Below let $n$ be fixed and large enough such that \eqref{iso88} holds.

By Lemma \ref{lem202}, there exists some $u_n\in\mathcal R_{\varepsilon}$ such that
\begin{equation}\label{iso10}
\int_Du_n\mathcal Gw_n dx\geq \int_Dw\mathcal Gw_n dx\quad\forall\,w\in\mathcal R_{\varepsilon}.
\end{equation}
Taking into account \eqref{iso88}, \eqref{iso10}, and choosing $B_2=B_{\bar r}(\bar x),$ $B_1=B_{r_\varepsilon}(\bar x)$, $u=u_n$ in Lemma \ref{fjin},
we get
\begin{equation}\label{iso30}
{\rm supp}(u_n)\subset \overline{B_{r_\varepsilon}(\bar x)}.
\end{equation}
Note that \eqref{fjin4} is satisfied  since supp$(\omega)\subset\overline{B_{r_\varepsilon}(\bar x)}$ for any $\omega\in\mathcal M_\varepsilon$ by the definition of $r_\varepsilon.$
Hence we obtain
\begin{equation}\label{iso50}
u_n\in\mathcal K_{\varepsilon}.
\end{equation}

To get a contradiction, we prove the following statement:
\begin{equation}\label{mkk}
\mbox{$E(u_n)\geq E(w_n),$  and the equality holds if and only if $u_n=w_n$.}
\end{equation}
In fact,  since $w_n\in\mathcal R_{\varepsilon}$ (see \eqref{iso5}), we get from \eqref{iso10} that
\begin{equation}\label{iso70}
\int_Du_n\mathcal Gw_n dx\geq \int_Dw_n\mathcal Gw_n dx.
\end{equation}
Now we calculate $E(u_n)-E(w_n)$ as follows
\begin{align}
E(u_n)-E(w_n)&=\frac{1}{2}\int_Du_n\mathcal Gu_n-w_n\mathcal Gw_n dx\label{iso202}\\
&=\frac{1}{2}\int_D(u_n-w_n)\mathcal G(u_n+w_n) dx \label{iso203}\\
&=\frac{1}{2}\int_D(u_n-w_n)\mathcal G(u_n-w_n) dx+\int_D(u_n-w_n)\mathcal Gw_n dx\label{iso204}\\
&\geq \int_D(u_n-w_n)\mathcal Gw_n dx \label{iso205}\\
&\geq 0. \label{iso206}
\end{align}
Note that in \eqref{iso202} we used the symmetry of the Green operator $\mathcal G,$ in \eqref{iso205} we used the  positivity of $\mathcal G$, and in \eqref{iso206} we used \eqref{iso70}. Moreover, it is easy to see that \eqref{iso206} is an equality if and only if $u_n=w_n$. Therefore $E(u_n)=E(w_n)$ if and only if $u_n=w_n$. This proves \eqref{mkk}.

Based on \eqref{mkk}, we can easily get a contradiction. In fact, since $u_n\in\mathcal K_{\varepsilon},$ we get from \eqref{iso678} that
$E(u_n)\leq M_{\varepsilon}$. Taking into account \eqref{iso55} we obtain $E(w_n)\geq E(u_n).$ This together with \eqref{mkk} gives $E(w_n)=E(u_n)$, and thus $w_n=u_n$, which contradicts \eqref{iso555}.

\end{proof}

\begin{remark}
Repeating the argument  in  Proposition \ref{isolated1}, we can in fact prove the following conclusion.
Let $\mathcal R$ be the rearrangement class of some nonnegative function $f_0\in L^\infty(D)$ and  $\mathcal M\subset\mathcal R$ be nonempty. Suppose there exist two open ball $B_1,B_2$ with $B_1\subset\subset B_2\subset\subset D$ such that
\begin{itemize}
\item[(i)]$\mathcal M$ is compact in $L^p(D)$, where $1<p<+\infty$ is fixed;
  \item [(ii)]  $\mathcal M\subset\mathcal K_1,\quad\mbox{where }\mathcal K_1=\{w\in\mathcal R\mid \mbox{supp}(w)\subset \bar{B}_1\};$
  \item [(iii)] $\mathcal M=\{w\in\mathcal K_2\mid E(w)=\sup_{\mathcal K_2} E\},$ where $\mathcal K_2=\{w\in\mathcal R\mid \mbox{supp}(w)\subset \bar{B}_2\};$
  \item [(iv)] for any $\omega\in\mathcal M,$ it holds that
  \[\sup_{ D\setminus B_2}\mathcal G\omega<\inf_{{ B_1}}\mathcal G\omega.\]
\end{itemize}
Then $\mathcal M$ must be an isolated set of local maximizers of $E$ over $\mathcal R$.
\end{remark}

\section{Proof of Theorem \ref{el2k2}: $k=2$ }
In this section, we prove Theorem \ref{el2k2} for  $k=2$ with  $\kappa_1\kappa_2<0$. The process is parallel to that in Section 4, therefore we may omit some similar arguments to avoid verbosity.

Throughout this section let $1<p<+\infty, M>0, \kappa_1>0,\kappa_2<0$ be fixed, and
 $(\bar x_1,\bar x_2)\in \mathbb D^2$ be a given isolated local minimum point of the corresponding Kirchhoff-Routh function
\begin{equation}\label{krf000}
{ W}(x_1,x_2)=- 2\kappa_1\kappa_2G(x_1,x_2)+ \kappa_1^2h(x_1,x_1)+\kappa_2^2h(x_2,x_2).
\end{equation}
Fix a small positive number $\bar r$ such that
\begin{itemize}
\item[(i)]$\overline{B_{\bar r}(\bar x_i)}\subset D$, $i=1,2$;
\item[(ii)]$\overline{B_{\bar r}(\bar x_1)}\cap \overline{B_{\bar r}(\bar x_2)}=\varnothing$;
\item[(iii)]$(\bar x_1,\bar x_2)$ is the unique minimum point of $W$ in $\overline{B_{\bar r}(\bar x_1)}\times \overline{B_{\bar r}(\bar x_2)}$.
\end{itemize}

Let $\Pi^1,\Pi^2:(0,\bar r)\to\Xi$ be two maps  satisfying
   \begin{equation}\label{h188}
 \begin{split}
& \lim_{\varepsilon\to0}\int_{\mathbb R^2}{\Pi^i_\varepsilon}(x)dx=|\kappa_i|,\\
& {\mbox{supp($\Pi^i_\varepsilon$)}}\subset \overline{B_\varepsilon(\mathbf 0)}
\,\,\mbox{ for any } \varepsilon\in(0,\bar r),\\
& \|\Pi^i_\varepsilon\|_{L^\infty(\mathbb R^2)}\leq M\varepsilon^{-2},
\end{split}
\end{equation}
where $i=1,2.$ For simplicity, denote
\[\kappa_{1,\varepsilon}=\int_{\mathbb R^2}{\Pi^1_\varepsilon}(x)dx,\quad \kappa_{2,\varepsilon}=\int_{\mathbb R^2}{\Pi^2_\varepsilon}(x)dx.\]

The $\mathcal K_\varepsilon$ defined by \eqref{kvar} now becomes
 \[\mathcal K_{\varepsilon}=\left\{w=\sum_{i=1}^2w_i\mid {\rm supp}(w_i)\subset \overline{B_{\bar r}(\bar x_i)},\,\,w_i\mbox{ is a rearrangement of sgn$(\kappa_i)\Pi_\varepsilon^i$},\,\,i=1,2\right\}.\]
 Denote $\mathcal M_{\varepsilon}$ the set of maximizers of $E$ over $\mathcal K_{\varepsilon}.$

Theorem \ref{thmex} in this situation can be stated as follows.

 \begin{theorem}\label{thmex2}
 For any $\varepsilon\in(0,\bar r)$, $\mathcal M_{\varepsilon}$ is nonempty,  and the following assertions hold.
\begin{itemize}
\item[(i)] There exists some $\varepsilon_1>0$, such that for any $\varepsilon\in(0,\varepsilon_1)$,  $ \mathcal M_{\varepsilon}$ is a set of  steady solutions to the  vorticity  equation.
\item [(ii)] For  $i=1,2$, $\mbox{ supp}(\omega\mathbf 1_{B_{\bar r}(\bar x_i)})$ ``shrinks" to $\bar x_i$  uniformly as $\varepsilon\to0$. More precisely,
for any $\delta\in(0,\bar r),$ there exists some $\varepsilon_2>0,$  such that for any $\varepsilon\in(0,\varepsilon_2)$, it holds that
\[\mbox{\rm supp}(\omega\mathbf 1_{B_{\bar r}(\bar x_i)})\subset \overline{B_{\delta}(\bar x_i)},\quad\forall\,\omega \in\mathcal M_{\varepsilon}.\]
\end{itemize}
 \end{theorem}

Define
 \[\mathcal R_{\varepsilon}=\left\{w\in L^\infty(D) \mid w \mbox{ is the rearrangement of some }  v\in\mathcal K_{\varepsilon}\right\}.\]
Note that the definition of $\mathcal R_{\varepsilon}$ is reasonable since any two elements in $\mathcal K_{\varepsilon}$ have the same distribution function.

To prove the stability of $\mathcal M_\varepsilon$, it suffices show that $\mathcal M_{\varepsilon}$ is compact in $L^p(D)$, and is an isolated set of local maximizers of $E$ over $\mathcal R_{\varepsilon}.$

Compactness still relies on Lemma \ref{yyds}.

\begin{proposition}[Compactness]\label{ysss}
For any $\varepsilon\in(0,\bar r)$, $\mathcal M_{\varepsilon}$ is compact in $L^p(D)$.
\end{proposition}

\begin{proof}
Fix some sequence  $\{\omega_n\}_{n=1}^{+\infty}\subset \mathcal M_{\varepsilon}$.  Denote $\omega_{n,i}=\omega_n\mathbf 1_{B_{\bar r}(\bar x_i)},$ $i=1,2.$

Below let $i\in\{1,2\}$ be fixed. For $\{\omega_{n,i}\}_{n=1}^{+\infty}$, there exist a subsequence, still denoted by   $\{\omega_{n,i}\}_{n=1}^{+\infty}$, and some $\eta_i\in L^p(D)$, such that $\omega_{n,i}$ converges weakly to $\eta_i$ as $n\to+\infty.$ It is easy to check that
$\eta_i\in \bar{\mathcal K}_{i,\varepsilon},$ the weak closure of $\mathcal K_{i,\varepsilon}$ in $L^p(B_{\bar r}(\bar x_i))$ (or $L^p(D)$),
where
\[\mathcal K_{i,\varepsilon}=\left\{w\in L^\infty(D)\mid {\rm supp}(w)\subset B_{\bar r}(\bar x_i), \,\,w \mbox{ is the rearrangement of sgn}(\kappa_i)\Pi_\varepsilon^i\right\}.\]
A similar argument as in the proof of Proposition \ref{comlem} gives
 \[\int_{B_{\bar r}(\bar x_i)}\eta_i\mathcal G\eta_i dx\geq \int_{B_{\bar r}(\bar x_i)}v\mathcal G\eta_i dx\quad\forall\,v\in\bar{ \mathcal K}_{i,\varepsilon}.\]
By Lemma \ref{yyds} we get  $\eta\in\mathcal K_{i,\varepsilon}$, and thus
$\omega_{n,i}$ converges strongly to $\eta_i$ as $n\to+\infty.$

Denote $\eta=\eta_1+\eta_2$, then $\eta\in\mathcal K_{\varepsilon}$ and $\omega_{n}$ converges strongly to $\eta$ as $n\to+\infty.$ Hence  the compactness of $\mathcal M_{\varepsilon}$ is proved.

\end{proof}

\begin{lemma}[Separation property]\label{lemk22}
Define
\[r_{\varepsilon}=\inf\left\{r>0\mid \omega\mathbf 1_{B_{\bar r}(\bar x_i)}=0\mbox{ a.e. in } D\setminus B_{r}(\bar x_i)\,\,\mbox{for any }\omega\in\mathcal M_{\varepsilon} \mbox{ and } i=1,2\right\}.\]
Then there exists $\varepsilon_3>0$,  such that  for any $\varepsilon\in(0,\varepsilon_3)$, it holds that
\[\sup_{D\setminus B_{\bar r}(\bar x_1)}\mathcal G\omega<\inf_{ {B_{r_\varepsilon}(\bar x_1)}}\mathcal G\omega,\quad\inf_{D\setminus B_{\bar r}(\bar x_2)}\mathcal G\omega>\sup_{ {B_{r_\varepsilon}(\bar x_2)}}\mathcal G\omega,\quad\forall\,\omega\in \mathcal M_{\varepsilon}.\]

\end{lemma}

\begin{proof}
It is clear that $r_{\varepsilon}>0$  for any $\varepsilon\in(0,\varepsilon_0)$. Moreover, $r_{\varepsilon}\to0$ as $\varepsilon\to 0$  by (i) in Theorem \ref{thmex2}.

Below let $\omega\in\mathcal M_{\varepsilon}$ be fixed.  For any $x\in D\setminus B_{\bar r}(\bar x_1),$ we estimate $\mathcal G\omega(x)$ as follows
\begin{align}
\mathcal G\omega(x)&=\int_{B_{\bar r}(\bar x_1)}G(x,y)\omega(y)dy+\int_{B_{\bar r}(\bar x_2)}G(x,y)\omega(y)dy\label{xsw555}\\
&\leq \int_{B_{\bar r}(\bar x_1)}G(x,y)\omega(y)dy+\|G\|_{L^\infty(B_{\bar r}(\bar x_1)\times B_{\bar r}(\bar x_2))}|\kappa_{2,\varepsilon}|. \label{xsw31}
\end{align}
In \eqref{xsw555} we used the fact $\omega=\omega\mathbf 1_{B_{\bar r}(\bar x_1)}+\omega\mathbf 1_{B_{\bar r}(\bar x_2)}.$
Since $\overline{B_{\bar r}(\bar x_1)}\cap \overline{B_{\bar r}(\bar x_2)}=\varnothing$, we get
\[\|G\|_{L^\infty(B_{\bar r}(\bar x_1)\times B_{\bar r}(\bar x_2))}<+\infty.\]
The first term in \eqref{xsw31} can be estimated as follows
\begin{equation}\label{xsw311}
\begin{split}
\int_{B_{\bar r}(\bar x_1)}G(x,y)\omega(y)dy&=-\frac{1}{2\pi}\int_{B_{\bar r}(\bar x_1)}\ln|x-y|\omega(y)dy-\int_{B_{\bar r}(\bar x_1)}h(x,y)\omega(y)dy \\
&=-\frac{1}{2\pi}\int_{B_{r_{\varepsilon}}(\bar x_1)}\ln|x-y|\omega(y)dy-\int_{B_{\bar r}(\bar x_1)}h(x,y)\omega(y)dy\\
&\leq -\frac{1}{2\pi}\int_{B_{r_{\varepsilon}}(\bar x_1)}\ln(\bar r-r_\varepsilon)\omega(y)dy+\|h\|_{L^\infty(B_{\bar r}(\bar x_1)\times B_{\bar r}(\bar x_1))}|\kappa_{1,\varepsilon}|\\
&= -\frac{\kappa_{1,\varepsilon}}{2\pi}\ln(\bar r-r_{\varepsilon})+ \|h\|_{L^\infty(B_{\bar r}(\bar x_1)\times B_{\bar r}(\bar x_1))}\kappa_{1,\varepsilon}\\
&\to-\frac{|\kappa_{1}|}{2\pi}\ln\bar r+\|h\|_{L^\infty(B_{\bar r}(\bar x_1)\times B_{\bar r}(\bar x_2))}|\kappa_{1}|
\end{split}
\end{equation}
as $\varepsilon\to0.$
On the other hand, for any  $x\in \overline{ B_{ r_{\varepsilon}}(\bar x_1)},$ we have
\begin{align}
\mathcal G\omega(x)&=\int_{B_{\bar r}(\bar x_1)}G(x,y)\omega(y)dy+\int_{B_{\bar r}(\bar x_2)}G(x,y)\omega(y)dy\\
&\geq \int_{B_{\bar r}(\bar x_1)}G(x,y)\omega(y)dy-\|G\|_{L^\infty(B_{\bar r}(\bar x_1)\times B_{\bar r}(\bar x_2))}|\kappa_{2,\varepsilon}|. \label{xsw31111}
\end{align}
For first term in \eqref{xsw31111},
\begin{equation}\label{xsw3119}
\begin{split}
\int_{B_{\bar r}(\bar x_1)}G(x,y)\omega(y)dy&=-\frac{1}{2\pi}\int_{B_{\bar r}(\bar x_1)}\ln|x-y|\omega(y)dy-\int_{B_{\bar r}(\bar x_1)}h(x,y)\omega(y)dy \\
&=-\frac{1}{2\pi}\int_{B_{r_{\varepsilon}}(\bar x)}\ln|x-y|\omega(y)dy-\int_{B_{\bar r}(\bar x_1)}h(x,y)\omega(y)dy\\
&\geq -\frac{|\kappa_{1,\varepsilon}|}{2\pi}\ln(2r_{\varepsilon})-\|h\|_{L^\infty(B_{\bar r}(\bar x_1)\times B_{\bar r}(\bar x_1))}|\kappa_{1,\varepsilon}|\\
&\to+\infty
\end{split}
\end{equation}
as $\varepsilon\to0.$
Combining \eqref{xsw31}, \eqref{xsw311}, \eqref{xsw31111} and \eqref{xsw3119}, we see that
\[\sup_{D\setminus B_{\bar r}(\bar x_1)}\mathcal G\omega<\inf_{ {B_{r_{\varepsilon}}(\bar x_1)}}\mathcal G\omega,\quad\forall\,\omega\in \mathcal M_{\varepsilon}\]
provided that $\varepsilon$ is sufficiently small.
Similarly, for small $\varepsilon$ we have
\[\inf_{D\setminus B_{\bar r}(\bar x_2)}\mathcal G\omega>\sup_{ {B_{r_{\varepsilon}}(\bar x_2)}}\mathcal G\omega,\quad\forall\,\omega\in \mathcal M_{\varepsilon}.\]
Hence the proof is finished.

\end{proof}

\begin{lemma}\label{bfjin}
Let $U\subset \mathbb R^2$ be a bounded domain, $B_1,B_2,B_3,B_4$ be open balls such that $B_1\subset\subset B_2\subset\subset U$, $B_3\subset\subset B_4\subset\subset U$, and $\bar B_2\cap\bar B_4=\varnothing$. Let $f_1,f_2\in L^1(U)$ such that $f_1\geq 0,f_2\leq 0$ a.e. in $U,$ ${\rm supp}(f_1)\subset \bar{B}_2, $ ${\rm supp}(f_2)\subset \bar{B}_4$, and
\begin{equation}\label{bfjin4}
\mathcal L(\{x\in U\mid f_1(x)>0\})\leq \mathcal L(B_1),\quad \mathcal L(\{x\in U\mid f_2(x)<0\})\leq \mathcal L(B_3).
\end{equation}
Let  $\mathcal R$ be the rearrangement class of $f_1+f_2$.
Let $\psi\in C(\bar U)$ satisfy
\begin{equation}\label{bfjin1}
\inf_{B_1}\psi>\sup_{U\setminus B_2}\psi,\quad\sup_{B_3}\psi< \inf_{U\setminus B_4}\psi.
\end{equation}
If $u\in\mathcal R$ satisfies
\begin{equation}\label{bfjin44}
\int_Uu\psi dx\geq \int_Uw\psi dx,\quad\forall \,w\in\mathcal R,
\end{equation}
then  supp$(u_+)\subset \bar B_1$,  supp$(u_-)\subset \bar{B}_3$.
\end{lemma}
\begin{proof}
The proof is very similar to that of Lemma \ref{fjin}, we omit it therefore.
\end{proof}

\begin{proposition}[Isolatedness]\label{isolated2}
Let  $\varepsilon_3>0$ be determined in Lemma \ref{lemk22}. Then for any $\varepsilon\in(0,\varepsilon_3)$ ,
there exists some $\delta>0$ such that
\[\mathcal M_{\varepsilon}=\left\{w\in (\mathcal M_{\varepsilon})_\delta\cap \mathcal R_{\varepsilon}\,\,\bigg|\,\,E(w)=\sup_{w\in (\mathcal M_{\varepsilon})_\delta\cap \mathcal R_{\varepsilon}}E(w)\right\},\]
where $(\mathcal M_{\varepsilon})_\delta$ is the $\delta$-neighborhood of $\mathcal M_{ \varepsilon}$ in $L^p(D)$.
\end{proposition}

\begin{proof}
Let $r_\varepsilon$ be defined in Lemma \ref{lemk22}. Then for any $\varepsilon\in(0,\varepsilon_3),$ it holds that
\begin{equation}\label{yssd}
\sup_{D\setminus B_{\bar r}(\bar x_1)}\mathcal G\omega<\inf_{ {B_{r_\varepsilon}(\bar x_1)}}\mathcal G\omega,\quad\inf_{D\setminus B_{\bar r}(\bar x_2)}\mathcal G\omega>\sup_{ {B_{r_\varepsilon}(\bar x_2)}}\mathcal G\omega,\quad \forall\,\omega\in \mathcal M_{\varepsilon}.\end{equation}

Denote \begin{equation}\label{iso6780}
M_{\varepsilon}=\sup_{w\in\mathcal K_{\varepsilon}}E(w).
\end{equation}
Below we show that there exists some $\delta>0$, such that for any $w\in (\mathcal M_{\varepsilon})_\delta\cap \mathcal R_{\varepsilon}$ satisfying $E(w)\geq M_{\varepsilon}$, it holds that
$w\in \mathcal M_{\varepsilon}.$
 Suppose by contradiction that the statement is not true, then for any positive integer $n$,  there exists some $w_n$  such that
  \begin{equation}\label{isok1}
  w_n\in  (\mathcal M_{\varepsilon})_{\frac{1}{n}}\cap \mathcal R_{\varepsilon},
  \end{equation}
    \begin{equation}\label{isok2}
 E(w_n)\geq M_{\varepsilon},
   \end{equation}
    \begin{equation}\label{isok3}
 w_n\notin \mathcal M_{\varepsilon}.
  \end{equation}
  Below we deduce a contradiction from \eqref{isok1}, \eqref{isok2} and \eqref{isok3}.

Since $\mathcal M_{\varepsilon}$ is compact by Lemma \ref{ysss}, there exists some $\eta\in\mathcal M_{\varepsilon}$ such that  up to a subsequence
 \begin{equation}\label{isok10}
\lim_{n\to+\infty}\|w_n-\eta\|_{L^p(D)}=0.
\end{equation}
Consequently
 \begin{equation}\label{isok11}
\lim_{n\to+\infty}\|\mathcal Gw_n-\mathcal G\eta\|_{L^\infty(D)}=0.
\end{equation}
Since $\eta\in\mathcal M_{\varepsilon}$, taking into account \eqref{yssd}, we deduce that
\begin{equation}\label{isok12}
\sup_{D\setminus B_{\bar r}(\bar x_1)}\mathcal Gw_n<\inf_{ {B_{r_\varepsilon}(\bar x_1)}}\mathcal Gw_n,\quad\inf_{D\setminus B_{\bar r}(\bar x_2)}\mathcal Gw_n>\sup_{ {B_{r_\varepsilon}(\bar x_2)}}\mathcal Gw_n,\quad \forall\,\omega\in \mathcal M_{\varepsilon},
\end{equation}
provided that $n$ is large enough.

Below fix a large $n$ such that \eqref{isok12} holds.
By Lemma \ref{lem202}, there exists some $u_n\in\mathcal R_{\varepsilon}$ such that
\begin{equation}\label{isok101}
\int_Du_n\mathcal Gw_n dx\geq \int_Dw\mathcal Gw_n dx\quad\forall\,w\in\mathcal R_{\varepsilon}.
\end{equation}
From  \eqref{isok12} and \eqref{isok101}, we can apply Lemma \ref{bfjin} to get
\begin{equation}\label{isok30}
{\rm supp}(u_n\mathbf 1_{\{u_n>0\}})\subset \overline{B_{r_\varepsilon}(\bar x_1)},\quad
{\rm supp}(u_n\mathbf 1_{\{u_n<0\}})\subset \overline{B_{r_\varepsilon}(\bar x_2)}
\end{equation}
which implies
\begin{equation}\label{isok50}
u_n\in\mathcal K_{\varepsilon}.
\end{equation}

From \eqref{iso6780}, \eqref{isok2} and  \eqref{isok50} we see that \begin{equation}\label{xkk}
E(u_n)\leq M_{\varepsilon}\leq E(w_n).
\end{equation}
On the other hand, by \eqref{isok101}, using symmetry and positivity of the Green operator, we can prove the following assertion
\begin{equation}\label{xkk1}
\mbox{$E(u_n)\geq E(w_n),$  and the equality holds if and only if $u_n=w_n$.}
\end{equation}
By \eqref{xkk} and  \eqref{xkk1} we get $w_n=u_n$, and thus $E(w_n)=M_{\varepsilon}$. This means $w_n=u_n\in\mathcal M_{\varepsilon}$, a contradiction to \eqref{isok3}.

\end{proof}

\begin{remark}
From the proof of Proposition \ref{isolated2}, we in fact obtain the following conclusion.
Let   $B_1,B_2,B_3,B_4$ be open balls such that $B_1\subset\subset B_2\subset\subset D$,  $B_3\subset\subset B_4\subset\subset D$, and $\bar{B}_2\cap\bar{B}_4=\varnothing$. Let $f,g \in L^\infty(D)$ such that $f\geq 0, g\leq 0$ a.e. in $D$, supp$(f)\subset \bar{B}_1,$ supp$(g)\subset \bar{B}_3$.
Let $\mathcal R$ be the rearrangement class of $f+g$ in $D$ and  $\mathcal M\subset\mathcal R$ be nonempty. Suppose that $\mathcal M$ satisfies the following conditions:
\begin{itemize}
\item[(i)]$\mathcal M$ is compact in $L^p(D)$, where $1<p<+\infty$ is fixed;
  \item [(ii)]  $\mathcal M\subset\mathcal K_1$, where \[\mathcal K_1=\{w=w_1+w_2\mid w_1\in\mathcal R_{f},\,\,\mbox{supp}(w_1)\subset \bar{B}_1,\,\,w_2\in\mathcal R_{g},\,\,\mbox{supp}(w_2)\subset \bar{B}_3 \};\]
  \item [(iii)] $\mathcal M=\{w\in\mathcal K_2\mid E(w)=\sup_{\mathcal K_2} E\},$ where
  \[\mathcal K_2=\{w=w_1+w_2\mid w_1\in\mathcal R_{f},\,\,\mbox{supp}(w_1)\subset\bar{B}_2,\,\,w_2\in\mathcal R_{g},\,\,\mbox{supp}(w_2)\subset \bar{B}_4 \};\]
  \item [(iv)] for any $\omega\in\mathcal M,$ it holds that
  \[\sup_{ D\setminus B_2}\mathcal G\omega<\inf_{{ B_1}}\mathcal G\omega,\quad \inf_{ D\setminus B_4}\mathcal G\omega>\sup_{{ B_3}}\mathcal G\omega.\]
\end{itemize}
Then $\mathcal M$ must be an isolated set of local maximizers of $E$ over $\mathcal R$.
\end{remark}

\section{Proof of Theorem \ref{thmex}}
For completeness, we give the proof of Theorem \ref{thmex} in this section. The basic idea comes from  Turkington's paper \cite{T}. See also Elcrat-Miller \cite{EM}.

\subsection{Existence of a maximizer and its profile}

\begin{proposition}\label{proop1}
For any $\varepsilon\in(0,\bar r)$,  $\mathcal M_{\varepsilon}$ is nonempty.  Moreover, for any $\omega\in\mathcal M_{\varepsilon}$ and $i\in\{1,\cdot\cdot\cdot,k\},$ there exists some nondecreasing function $\phi_{\omega,i}:\mathbb R\to\mathbb R$  such that
\[{\rm sgn}(\kappa_i)\omega=\phi_{\omega,i}({\rm sgn}(\kappa_i)\mathcal G\omega) \,\, \mbox{ \rm  a.e. in } B_{\bar r}(\bar x_i).\]

\end{proposition}
\begin{proof}
Fix $\varepsilon\in(0,\bar r)$. Since $\mathcal K_{\varepsilon}$ is bounded in $L^2(D)$, it is easy to check that
\[M_{\varepsilon}=\sup_{w\in\mathcal K_{\varepsilon}}E(w)<+\infty.\]
Choose a sequence $\{w_n\}_{n=1}^{+\infty}\subset \mathcal K_{\varepsilon}$ such that
\begin{equation}\label{ropz1}
\lim_{n\to+\infty}E(w_n)= M_{\varepsilon}.
\end{equation}
Up to a subsequence, we can assume that $w_n$ converges weakly to some $\omega\in \bar{\mathcal K}_{\varepsilon}$ in $L^2(D)$ as $n\to+\infty$. Here we denote by $\bar{\mathcal K}_{\varepsilon}$  the weak closure of ${\mathcal K}_{\varepsilon}$  in $L^2(D).$ By \eqref{ropz1} and the weak sequential continuity of $E$
in $L^2(D),$ we have $E(\omega)=M_{\varepsilon},$ that is, $\omega$ is a maximizer of $E$ over $\bar {\mathcal K}_{\varepsilon}.$

Next we verify that $\omega\in\mathcal K_{\varepsilon}.$  Denote
\[\mathcal V_{ i}=\left\{v \in L^1(B_{\bar r}(\bar x_i))\mid  v \mbox{ is a rearrangement of sgn}(\kappa_i) \Pi^i_\varepsilon \mbox{ in }B_{\bar r}(\bar x_i) \right\},\,\,i=1,\cdot\cdot\cdot,k.\]
It suffices to show that $\omega_i:=\omega\mathbf 1_{B_{\bar r}(\bar x_i)}\in {\mathcal V}_{i}$ for any  $i\in\{1,\cdot\cdot\cdot,k\}.$
Let $\bar{\mathcal V}_{i}$ be the weak closure of ${\mathcal V}_{i}$  in $L^2(B_{\bar r}(\bar x_i)).$ It is clear that $\omega_i|_{B_{\bar r}(\bar x_i)}\in \bar{\mathcal V}_{i}$. Here  $\omega_i|_{B_{\bar r}(\bar x_i)}$ denotes the restriction of  $\omega_i$ in  $B_{\bar r}(\bar x_i)$.
We claim that for any  $v\in \bar{\mathcal V}_{i}$ and $s\in[0,1]$, it holds that
\begin{equation}\label{vp1}
\omega+s(v-\omega_i)\in \bar{\mathcal K}_{\varepsilon}.
\end{equation}
Note that we have regarded $v$ as a function in $D$ by setting  $v\equiv 0$ in $D\setminus B_{\bar r}(\bar x_i).$
To prove \eqref{vp1}, notice that   $\bar {\mathcal V}_{\varepsilon,i}$ is a convex set by   Lemma \ref{lem201}, thus $sv+(1-s)\omega_i\in \bar {\mathcal V}_{i}.$ Taking into account the definition of $\bar{\mathcal K}_{\varepsilon}$, we deduce that
 \[\omega-\omega_i+sv+(1-s)\omega_i\in \bar{\mathcal K}_{\varepsilon},\]
  which is exactly \eqref{vp1}.
Since $\omega$ is a maximizer of $E$ over $\bar{\mathcal K}_{\varepsilon}$, we get from  \eqref{vp1}   that
\[\frac{d}{ds}E(\omega+s(v-\omega_i))\bigg|_{s=0^+}\leq 0,\]
which yields
\begin{equation}\label{vp01}
\int_{B_{\bar r}(\bar x_i)}\omega_i\mathcal G\omega dx\geq \int_{B_{\bar r}(\bar x_i)}v\mathcal G\omega dx.
\end{equation}
Note that   \eqref{vp01}  holds for any $v\in\bar{\mathcal V}_{i}$. For the case of $\kappa_i$ being positive, we can apply Lemma \ref{yyds} by
taking therein
\[\Omega=B_{\bar r}(\bar x_i),\quad \mathcal R_{f_0}= \mathcal V_i,\quad \tilde f=\omega_i|_{B_{\bar r}(\bar x_i)},\quad g=\mathcal G\omega|_{B_{\bar r}(\bar x_i)}\]
 to get
\[\omega_i|_{B_{\bar r}(\bar x_i)}\in {\mathcal V}_{i},\quad  \omega_i =\phi_{\omega,i}(\mathcal G\omega) \mbox{ a.e. in }B_{\bar r}(\bar x_i)\]
 for some nondecreasing function $\phi_{\omega,i}:\mathcal R\to[-\infty,+\infty]$.
Similarly, if $\kappa_i<0$, we can take
\[\Omega=B_{\bar r}(\bar x_i),\quad \mathcal R_{f_0}=- \mathcal V_i,\quad \tilde f=-\omega_i|_{B_{\bar r}(\bar x_i)},\quad g=-\mathcal G\omega|_{B_{\bar r}(\bar x_i)}\]
 in Lemma \ref{yyds} to get
\[-\omega_i|_{B_{\bar r}(\bar x_i)}\in -{\mathcal V}_{i},\quad  -\omega_i =\phi_{\omega,i}(-\mathcal G\omega) \mbox{ a.e. in }B_{\bar r}(\bar x_i)\]
 for some nondecreasing function $\phi_{\omega,i}:\mathcal R\to[-\infty,+\infty]$.
In either case, since $\omega_i\in L^\infty(D)$, we can redefine $\phi_{\omega,i}$ such that it is real-valued. The proof is completed.

\end{proof}

In the rest of this section, denote $\omega_i=\omega\mathbf 1_{B_{\bar r}(\bar x_i)}$ for  $\omega\in\mathcal M_{\varepsilon}$.   Then it is obvious that $\omega=\sum_{i=1}^k\omega_i$. Define the $i$-th Lagrange multiplier  related to $\omega$ as follows
\[\mu_{\omega,i}=\inf\{s\in\mathbb R\mid \phi_{\omega,i}(s)>0\},\]
where $\phi_{\omega,i}$ is the nondecreasing function determined in Proposition \ref{proop1}.
From the definition of $\mu_{\omega,i}$,  we see that
\begin{equation}
{\rm sgn}(\kappa_i)\omega>0\quad \mbox{ a.e. in }\{x\in B_{\bar r}(\bar x_i)\mid {\rm sgn}(\kappa_i)\mathcal G\omega(x)> \mu_{\omega,i}\},
\end{equation}
\begin{equation}
{\rm sgn}(\kappa_i)\omega=0  \quad \mbox{ a.e. in }\{x\in B_{\bar r}(\bar x_i)\mid {\rm sgn}(\kappa_i)\mathcal G\omega(x)< \mu_{\omega,i}\}.
\end{equation}
On the level set $\{x\in B_{\bar r}(\bar x_i)\mid {\rm sgn}(\kappa_i)\mathcal G\omega(x)= \mu_{\omega,i}\},$ due to the fact that for any Sobolev function all its weak derivatives vanish  on its level sets (see p. 153, \cite{EV}), we have ${\rm sgn}(\kappa_i)\omega=-{\rm sgn}(\kappa_i)\Delta \mathcal G\omega=0$ a.e.. To conclude, we have shown that
\begin{equation}\label{nbd103}
\left\{\mathbf x\in B_{\bar r}(\bar x_i)\mid {\rm sgn}(\kappa_i)\omega(x)>0\right\}=\left\{x\in B_{\bar r}(\bar x_i)\mid {\rm sgn}(\kappa_i)\mathcal G\omega(x)> \mu_{\omega,i}\right\}.
\end{equation}
For  convenience, denote for any $\omega\in\mathcal M_{\varepsilon}$
\begin{equation}\label{vcor}
A_{\omega,i}=\left\{\mathbf x\in B_{\bar r}(\bar x_i)\mid {\rm sgn}(\kappa_i)\omega(x)>0\right\},
\end{equation}
called the $i$-th vortex core related to $\omega$.
By \eqref{h13}, it is clear that \begin{equation}\label{lala22}
\mathcal L(A_{\omega,i})\leq \pi\varepsilon^2,\quad \forall\,\omega\in\mathcal M_{\varepsilon}.
\end{equation}

\begin{remark}
By now we do not know whether $\omega$ is a steady  solution to the vorticity equation.
\end{remark}

\subsection{Size of the vortex cores}

The aim of this subsection is to prove the following proposition concerning the size of the vortex cores defined by \eqref{vcor}.
\begin{proposition}\label{proop2}
There exists some $R_0>0$, not depending  on $\varepsilon$, such that for any $\varepsilon\in(0,\bar r)$, it holds that
\[ {\rm diam(}A_{\omega,i})\leq R_0\varepsilon,\quad\forall\,\omega\in\mathcal M_{\varepsilon},\,\,i\in\{1,\cdot\cdot\cdot,k\}.\]
\end{proposition}

Before giving the proof of  Proposition \ref{proop2}, we need some necessary asymptotic estimates for the maximizers as $\varepsilon\to0$.

Below we use $C$ to denote various positive numbers that do not depend  on $\varepsilon$, but possibly depend on $D,(\bar x_1,\cdot\cdot\cdot,\bar x_k),\Pi^1,\cdot\cdot\cdot,\Pi^k$ and $\vec\kappa$.
We also denote for each $i\in\{1,\cdot\cdot\cdot,k\}$
\[\kappa_{\varepsilon,i}=\int_D\omega_i dx.\]
Then by \eqref{hh13} it holds that
\[\lim_{\varepsilon\to0}\kappa_{\varepsilon,i}=\kappa_i.\]
Without loss of generality, we assume that
  \begin{equation}\label{wb17}
  0<\inf_{\varepsilon\in(0,\bar r)}\kappa_{\varepsilon,i}\leq \sup_{\varepsilon\in(0,\bar r)}\kappa_{\varepsilon,i}<+\infty.
  \end{equation}

\begin{lemma}\label{lbd}
It holds that
\begin{equation}\label{lala5}
E(\omega)\geq -\frac{\ln \varepsilon}{4\pi} \sum_{i=1}^k\kappa^2_{\varepsilon,i}-C,\quad \forall\,\omega\in\mathcal M_{\varepsilon}.
\end{equation}

\end{lemma}
\begin{proof}
Fix $\omega\in \mathcal M_{\varepsilon}$.
Define
\[w(x)=\sum_{i=1}^kw_i,\quad w_i(x)=\Pi^i_\varepsilon(x-\bar x_i).\]
It is easy to check  that $w\in\mathcal K_{\varepsilon}$, hence
\begin{equation}\label{wb10}
E(\omega)\geq E(w).
\end{equation}
To estimate $E(w),$ we write it as follows:
\begin{equation}\label{wb11}
\begin{split}
E(w)=&\frac{1}{2}\int_{D}\int_{D}G(x,y)w(x)w(y)dxdy\\
=&-\frac{1}{4\pi}\int_{D}\int_{D}\ln|x-y|w(x)w(y)dxdy-\frac{1}{2}\int_{D}\int_{D}h(x,y)w(x)w(y)dxdy \\
=&-\frac{1}{4\pi}\sum_{i=1}^k\int_{D}\int_{D}\ln|x-y|w_i(x)w_i(y)dxdy\\
&-\frac{1}{2\pi}\sum_{1\leq i<j\leq k}\int_{D}\int_{D}\ln|x-y|w_i(x)w_j(y)dxdy\\
&-\frac{1}{2}\sum_{i,j=1}^k\int_{D}\int_{D}h(x,y)w_i(x)w_j(y)dxdy \\
=&-\frac{1}{4\pi}\sum_{i=1}^k\int_{B_{\varepsilon}(\mathbf 0)}\int_{B_{\varepsilon}(\mathbf 0)}\ln|x-y|\Pi^i_\varepsilon(x)\Pi^i_\varepsilon(y)dxdy\\
&-\frac{1}{2\pi}\sum_{1\leq i<j\leq k}\int_{B_{\bar r}(\bar x_i)}\int_{B_{\bar r}(\bar x_j)}\ln|x-y|w_i(x)w_j(y)dxdy\\
&-\frac{1}{2}\sum_{i,j=1}^k\int_{B_{\bar r}(\bar x_i)}\int_{B_{\bar r}(\bar x_j)}h(x,y)w_i(x)w_j(y)dxdy.
\end{split}
\end{equation}
For the first term, since $|x-y|\leq 2\varepsilon$ for any $x,y\in B_{\varepsilon}(\mathbf 0)$, we have
\begin{equation}\label{wb12}
-\frac{1}{4\pi}\sum_{i=1}^k\int_{B_{\varepsilon}(\mathbf 0)}\int_{B_{\varepsilon}(\mathbf 0)}\ln|x-y|\Pi^i_\varepsilon(x)\Pi^i_\varepsilon(y)dxdy\geq -\frac{\ln(2\varepsilon)}{4\pi}\sum_{i=1}^k\kappa^2_{\varepsilon,i}.
\end{equation}
For the second term, by \eqref{br2} we have
\[0<  \inf_{x\in B_{\bar r}(\bar x_i),y\in B_{\bar r}(\bar x_j)}|x-y|<\sup_{x\in B_{\bar r}(\bar x_i),y\in B_{\bar r}(\bar x_j)}|x-y|< +\infty,\quad \forall\,1\leq i<j\leq k,\]
from which we can easily deduce that
\begin{equation}\label{wb13}
-\frac{1}{2\pi}\sum_{1\leq i<j\leq k}\int_{B_{\bar r}(\bar x_i)}\int_{B_{\bar r}(\bar x_j)}\ln|x-y|w_i(x)w_j(y)dxdy\geq -C\sum_{1\leq i<j}|\kappa_{\varepsilon,i}\kappa_{\varepsilon,j}|.
\end{equation}
For the third term,   by \eqref{br1}, \eqref{br2} and the fact that $h\in C(D\times D)$ we have
\begin{equation}\label{uui1}
\sup_{x\in B_{\bar r}(\bar x_i),y\in B_{\bar r}(\bar x_j)}|h(x,y)|<+\infty,\quad \forall\,1\leq i,j\leq k,
\end{equation}
therefore it is   clear that
\begin{equation}\label{wb14}
-\frac{1}{2}\sum_{i,j=1}^k\int_{B_{\bar r}(\bar x_i)}\int_{B_{\bar r}(\bar x_j)}h(x,y)w_i(x)w_j(y)dxdy \geq -C\sum_{i,j=1}^k|\kappa_{\varepsilon,i}\kappa_{\varepsilon,j}|.
\end{equation}
Combining \eqref{wb10}-\eqref{wb14}, we get
\begin{equation}\label{lala5d}
E(\omega)\geq -\frac{\ln \varepsilon}{4\pi} \sum_{i=1}^k\kappa^2_{\varepsilon,i}-C\sum_{i,j=1}^k|\kappa_{\varepsilon,i}\kappa_{\varepsilon,j}|.
\end{equation}
Taking into account \eqref{wb17}, we get the desired estimate \eqref{lala5}.
\end{proof}

Below we fix a positive number $\hat r>\bar r$ such that
\begin{equation}\label{brt1}
\overline{B_{\hat r}(\bar x_i)}\subset D,\quad\forall\,1\leq i\leq k,
\end{equation}
\begin{equation}\label{brt2}
\overline{B_{\hat r}(\bar x_i)}\cap \overline{B_{\hat r}(\bar x_j)}=\varnothing,\quad\forall\,1\leq i< j\leq k.
\end{equation}
This is doable by \eqref{br1} and \eqref{br2}.

\begin{lemma}\label{xqt1}
It holds that
\begin{equation}\label{nbd167}
\|\mathcal G\omega\|_{L^\infty(\partial B_{\hat r}(\bar x_i))} \leq  C,\quad\forall\,\omega\in \mathcal M_\varepsilon,\,i\in\{1,\cdot\cdot\cdot,k\}.
 \end{equation}
\end{lemma}
\begin{proof}
Fix $\omega\in \mathcal M_\varepsilon$ and $i\in\{1,\cdot\cdot\cdot,k\}.$
 For any $x\in  \partial B_{\hat r}(\bar x_i)$, we have
\begin{align*}
\mathcal G\omega(x)=&-\frac{1}{2\pi}\int_D\ln|x-y|\omega(y)dy-\int_Dh(x,y)\omega(y)dy\\
=&-\frac{1}{2\pi}\sum_{j=1}^k\int_{B_{\bar r}(\bar x_j)}\ln|x-y|\omega_j(y)dy-\sum_{j=1}^k\int_{B_{\bar r}(\bar x_j)}h(x,y)\omega_j(y)dy.
\end{align*}
Taking into account  \eqref{br1}, \eqref{br2}, the fact that $h\in C(D\times D)$,  and the following observation
 \[ \hat r-\bar r\leq |x-y|\leq {\rm diam}(D), \quad \forall\,x \in \partial B_{\hat r}(\bar x_i),\,y\in \cup_{j=1}^kB_{\bar r}(\bar x_j),\]
we get \eqref{nbd167}   immediately.

\end{proof}

\begin{lemma}\label{xqt0}
For any $\omega\in\mathcal M_\varepsilon$ and $i\in  \{1,\cdot\cdot\cdot,k\},$ it holds that
\begin{equation}\label{nico1}
{\rm sgn}(\kappa_i)\mathcal G\omega(x)\geq -C,\quad \forall\,x\in B_{\bar r}(\bar x_i).
\end{equation}
\end{lemma}

\begin{proof}
Fix $\omega\in\mathcal M_\varepsilon$ and $i\in  \{1,\cdot\cdot\cdot,k\}.$   Without loss of generality, assume that $\kappa_i>0$. For any $x\in B_{\bar r}(\bar x_i)$, notice that
\begin{align*}
\mathcal G\omega(x)=&\int_DG(x,y)\omega(y)dxdy\\
 =&-\frac{1}{2\pi}\int_D\ln|x-y|\omega(y)dy-\int_Dh(x,y)\omega(y)dy\\
=&-\frac{1}{2\pi}\int_{B_{\bar r}(\bar x_i)}\ln|x-y|\omega_i(y)dy-\frac{1}{2\pi}\sum_{j\neq i}\int_{B_{\bar r}(\bar x_j)}\ln|x-y|\omega_j(y)dy\\
&-\sum_{j=1}^k\int_{B_{\bar r}(\bar x_j)}h(x,y)\omega_j(y)dy.
\end{align*}
For the first term, since $|x-y|\leq 2\bar r$ for any $x,y\in B_{\bar r}(\bar x_i),$ we have
\[-\frac{1}{2\pi}\int_{B_{\bar r}(\bar x_i)}\ln|x-y|\omega_i(y)dy\geq -\frac{\kappa_{\varepsilon,i}}{2\pi}\ln(2\bar r).\]
The other two terms are uniformly bounded as in the proof of Lemma \ref{xqt1}. Hence \eqref{nico1} is proved.
\end{proof}

\begin{lemma}\label{xqt2}
It holds that
\begin{equation}\label{nbd12}
\mu_{\omega,i}\geq -C, \quad \forall\,\omega\in\mathcal M_\varepsilon,\,\,i\in  \{1,\cdot\cdot\cdot,k\}.
\end{equation}
\end{lemma}
\begin{proof}
In this proof  let $C$ be the positive number in \eqref{nico1}. Suppose by contradiction there exist some $\omega\in\mathcal M_\varepsilon$ and $i\in  \{1,\cdot\cdot\cdot,k\}$ such that
\[\mu_{\omega,i}< -C.\]
Then by \eqref{nico1} we have
\[{\rm sgn}(\kappa_i)\mathcal G\omega(x)> \mu_{\omega,i}\quad \forall\,x\in B_{\bar r}(\bar x_i),\]
or equivalently,
\begin{equation}\label{nbd10}
\left\{x\in B_{\bar r}(\bar x_i)\mid {\rm sgn}(\kappa_i)\mathcal G\omega(x)> \mu_{\omega,i}\right\}=B_{\bar r}(\bar x_i).
\end{equation}
This together with \eqref{nbd103} yields
\begin{equation}
\left\{ x\in B_{\bar r}(\bar x_i)\mid {\rm sgn}(\kappa_i)\omega(x)>0\right\}=B_{\bar r}(\bar x_i),
\end{equation}
which is impossible since
\[\mathcal L(\left\{  x\in B_{\bar r}(\bar x_i)\mid {\rm sgn}(\kappa_i)\omega(x)>0\right\})=\mathcal L(\{x\in\mathbb R^2\mid \Pi_\varepsilon(x)>0\})\leq \pi\varepsilon^2<\pi \bar r^2.\]
\end{proof}

Combining Lemma \ref{xqt1} and Lemma \ref{xqt2}, we obtain
\begin{lemma}\label{xqt3}
For any $\omega\in \mathcal M_\varepsilon$ and $i\in\{1,\cdot\cdot\cdot,k\},$ it holds that
\begin{equation}
 {\rm sgn}(\kappa_i)\mathcal G\omega(x)-\mu_{\omega,i}\leq  C,\quad\forall\,x\in \partial B_{\hat r}(\bar x_i).
 \end{equation}
\end{lemma}

Below for $\omega\in\mathcal M_\varepsilon,$ denote
\begin{equation}\label{dotou}
T_\omega= \sum_{i=1}^k\int_D{\rm sgn}(\kappa_i)\omega_i\left({\rm sgn}(\kappa_i)\mathcal G\omega-\mu_{\omega,i}\right)dx.
\end{equation}
\begin{lemma}\label{ubd}
It holds that
\begin{equation}\label{doto}
 T_\omega\leq C,\quad \forall\,\omega\in\mathcal M_\varepsilon.
\end{equation}
\end{lemma}

\begin{proof}
Fix $\omega\in\mathcal M_{ \varepsilon}$ and $i\in\{1,\cdot\cdot\cdot,k\}$.
It suffices to show
\[  T_{\omega,i}:=\int_D{\rm sgn}(\kappa_i)\omega_i\left({\rm sgn}(\kappa_i)\mathcal G\omega-\mu_{\omega,i}\right)dx\leq C.\]
For simplicity  we assume that $\kappa_i>0$.
By Lemma \ref{xqt3}, there exists $ \mu_0>0$, not depending on $\varepsilon,$ such that
\begin{equation}\label{wb204}
\mathcal G\omega-\mu_{\omega,i}-\mu_0<0 \quad\mbox{on }\partial B_{\hat r}(\bar x_i).
\end{equation}
Denote $\zeta_i=\mathcal G\omega-\mu_{\omega,i}-\mu_0. $ Since
\[T_{\omega,i}=\int_D\omega_i\zeta_i dx+\mu_0\kappa_{\varepsilon,i}\leq \int_D\omega_i(\zeta_i)_+ dx+\mu_0\kappa_{\varepsilon,i},\]
we need only to show that
\begin{equation}\label{zett9}
\int_D\omega_i(\zeta_i)_+dx\leq C.
\end{equation}
For convenience, denote \[V_i=\{x\in B_{\bar r}(\bar x_i)\mid \zeta_i(x)>0\}.\]
It is easy to see that $V_i\subset A_{\omega,i}$, which together with \eqref{lala22} yields
\[\mathcal L(V_i)\leq \mathcal L(A_{\omega,i})\leq \pi\varepsilon^2.\]
On the one hand, by integration by parts we have
\begin{equation}\label{lalai00}
 \int_D\omega_i(\zeta_i)_+ dx=\int_{V_i}(-\Delta\zeta_i)\zeta_i dx=\int_{B_{\hat r}(\bar x_i)}(-\Delta(\zeta_i)_+)(\zeta_i)_+ dx=\int_{B_{\hat r}(\bar x_i)}|\nabla (\zeta_i)_+|^2 dx.
\end{equation}
Here we used the fact  that    $(\zeta_i)_+$ vanishes on $\partial B_{\hat r}(\bar x_i)$.
On the other hand,
\begin{equation}\label{hooo}
\begin{split}
 \int_D\omega_i(\zeta_i)_+ dx&= \int_{V_i}\omega_i(\zeta_i)_+ dx\\
 &\leq \sqrt{\mathcal L (V_i)}\|\omega_i\|_{L^\infty(D)}\|(\zeta_i)_+\|_{L^2(B_{\bar r}(\bar x_i))}\\
 &\leq C\varepsilon^{-1}\left(\|(\zeta_i)_+\|_{L^1(B_{\bar r}(\bar x_i))}+\|\nabla(\zeta_i)_+\|_{L^1(B_{\bar r}(\bar x_i))}\right)\\
  &= C\varepsilon^{-1}\left(\|(\zeta_i)_+\|_{L^1(V_{i})}+\|\nabla(\zeta_i)_+\|_{L^1(V_{i})}\right)\\
   &\leq C\varepsilon^{-1}\sqrt{\mathcal L(V_i)}\left(\|(\zeta_i)_+\|_{L^2(V_{i})}+\|\nabla(\zeta_i)_+\|_{L^2(V_{i})}\right)\\
    &\leq C\left(\|(\zeta_i)_+\|_{L^2(B_{\hat r}(\bar x_i))}+\|\nabla(\zeta_i)_+\|_{L^2(B_{\hat r}(\bar x_i))}\right)\\
 &\leq C \|\nabla(\zeta_i)_+\|_{L^2(B_{\hat r}(\bar x_i))}.
 \end{split}
 \end{equation}
 Note that in the first equality we used the fact that $\omega(\zeta_{i})_{+}$ vanishes off the set $V_{i}$,
in the first and the third inequality we used  H\"older's inequality,  in the second inequality we used the Sobolev embedding $W^{1,1}(B_{\bar r}(\bar x_i))\hookrightarrow L^2(B_{\bar r}(\bar x_i))$, and in the last inequality we used the Poincar\'e inequality.
Combining \eqref{lalai00} and \eqref{hooo}, we obtain the desired estimate \eqref{zett9}. Hence the proof is finished.
\end{proof}

\begin{lemma}\label{s1mp}
It holds that
\begin{equation}\label{o1p1}
\sum_{i=1}^k\mu_{\omega,i}|\kappa_{\varepsilon,i}|\geq -\frac{\ln\varepsilon}{2\pi}\sum_{i=1}^k\kappa^2_{\varepsilon,i}-C,\quad\forall\,\omega\in\mathcal M_\varepsilon,\,i\in\{1,\cdot\cdot\cdot,k\}.
\end{equation}
\end{lemma}

\begin{proof}
By the definition of $T_\omega$ (see \eqref{dotou}) we have the relation
\[T_\omega=2E(\omega)-\sum_{i=1}^k\mu_{\omega,i}|\kappa_{\varepsilon,i}|.\]
Then \eqref{o1p1} follows immediately from Lemma  \ref{lbd} and Lemma  \ref{ubd}.

\end{proof}
 
\begin{lemma}\label{s1mp6}
It holds that
\begin{equation}\label{o1p6}
\mu_{\omega,i} \leq -\frac{\ln\varepsilon}{2\pi} |\kappa_{\varepsilon,i}|+C,\quad\forall\,\omega\in\mathcal M_\varepsilon,\,i\in\{1,\cdot\cdot\cdot,k\}.
\end{equation}
\end{lemma}

\begin{proof}
 Fix $\omega\in\mathcal M_\varepsilon$ and $i\in\{1,\cdot\cdot\cdot,k\}.$ For any $x\in A_{\omega,i},$ by \eqref{nbd103} and \eqref{vcor}, we have that 
\begin{equation}\label{ego0}
\mu_{\omega,i}\leq {\rm sgn}(\kappa_{i})\mathcal G\omega(x).
\end{equation}
To proceed, we estimate ${\rm sgn}(\kappa_i)\mathcal G\omega(x)$ as follows:
\begin{equation}\label{ego1}
\begin{split}
  &{\rm sgn}(\kappa_i)\mathcal G\omega(x)\\
   =&{\rm sgn}(\kappa_i)\sum_{j=1}^k\int_DG(x,y)\omega_j(y)dy\\
   =&{\rm sgn}(\kappa_i)\int_D\left(-\frac{1}{2\pi}\ln|x-y|-h(x,y)\right)\omega_i(y)dy +{\rm sgn}(\kappa_i)\sum_{j\neq i}\int_DG(x,y)\omega_j(y)dy\\
 \leq &\int_{B_\varepsilon (\mathbf 0)}-\frac{1}{2\pi}\ln|y|\Pi^i_\varepsilon(y) dy+C \\
=&   \frac{1}{2\pi}\int_{B_\varepsilon (\mathbf 0)}\ln\frac{\varepsilon}{|y|}\Pi^i_\varepsilon(y) dy-\frac{\ln\varepsilon}{2\pi} |\kappa_{\varepsilon,i}|+C\\
 \leq &\frac{1}{2\pi} M\varepsilon^{-2}\int_{  B_{\varepsilon}(\mathbf 0)}\ln\frac{\varepsilon}{|y|} dy -\frac{\ln\varepsilon}{2\pi} |\kappa_{\varepsilon,i}|+C\\
 =& \frac{1}{2\pi}M \int_{  B_{1}(\mathbf 0)}\ln\frac{1}{|y|} dy -\frac{\ln\varepsilon}{2\pi} |\kappa_{\varepsilon,i}|+C\\
 \leq& -\frac{\ln\varepsilon}{2\pi} |\kappa_{\varepsilon,i}|+C.
\end{split}
\end{equation}
 Note that in the first inequality of \eqref{ego1} we used Lemma \ref{rri1}  and the following facts
 \[\left|\int_Dh(x,y)\omega_i(y)dy\right|\leq |\kappa_{\varepsilon,i}|\sup_{x',y\in B_{\bar r}(\bar x_i)}|h(x',y)|\leq C \quad (\mbox{by the continuity of $h$}),\]
 \[\left|\int_DG(x,y)\omega_j(y)dy\right|\leq |\kappa_{\varepsilon,j}|\sup_{x'\in B_{\bar r}(\bar x_i), y\in B_{\bar r}(\bar x_j) }|G(x',y)|\leq C,\,\,\forall j\neq i \quad(\mbox{by \eqref{br1} and \eqref{br2}}),\]
and in the second inequality we used \eqref{hhhh13}.
 Hence the required estimate is proved.
\end{proof}

A straightforward corollary of Lemma \ref{s1mp} and Lemma \ref{s1mp6} is the following lower bound for each $|\kappa_{\varepsilon,i}|\mu_{\omega,i}$.

\begin{lemma}\label{s1mp7}
It holds that
\begin{equation}\label{o1p7}
|\kappa_{\varepsilon,i}|\mu_{\omega,i} \geq -\frac{\ln\varepsilon}{2\pi} \kappa_{\varepsilon,i}^2-C,\quad\forall\,\omega\in\mathcal M_\varepsilon,\,i\in\{1,\cdot\cdot\cdot,k\}.
\end{equation}
\end{lemma}
\begin{proof}
Fix $\omega\in\mathcal M_\varepsilon$ and $i\in\{1,\cdot\cdot\cdot,k\}$. Then by Lemma \ref{s1mp} and Lemma \ref{s1mp6},
\begin{align*}
|\kappa_{\varepsilon,i}|\mu_{\omega,i} & =\sum_{j=1}^{k}|\kappa_{\varepsilon,j}|\mu_{\omega,j}-\sum_{j\neq i}|\kappa_{\varepsilon,j}|\mu_{\omega,j}\\
&\geq  -\frac{\ln\varepsilon}{2\pi}\sum_{j=1}^k\kappa^2_{\varepsilon,j}-C-\sum_{j\neq i}|\kappa_{\varepsilon,j}|\left(-\frac{\ln\varepsilon}{2\pi} |\kappa_{\varepsilon,j}|+C\right)\\
&\geq -\frac{\ln\varepsilon}{2\pi} \kappa_{\varepsilon,i}^2-C.
\end{align*}

\end{proof}

Now we are ready to prove Proposition \ref{proop2}.

\begin{proof}[Proof of Proposition \ref{proop2}]
Fix $\omega\in\mathcal M_{ \varepsilon}$ and $i\in\{1,\cdot\cdot\cdot,k\}$.  By \eqref{nbd103} and \eqref{vcor},  for any $x\in A_{\omega,i}$,  it holds that
\begin{equation}\label{uyi}
{\rm sgn}(\kappa_i)\mathcal G\omega(x)\geq \mu_{\omega,i}.
\end{equation}
In view of \eqref{br1}, \eqref{br2} and the continuity of $h$, we can rewrite \eqref{uyi} as follows:
\begin{equation}\label{uyii}
-\frac{1}{2\pi}\int_D\ln|x-y||\omega_i(y)|dy\geq \mu_{\omega,i}-C,
\end{equation}
which further implies
\begin{equation}\label{uyiii}
-\frac{|\kappa_{\varepsilon,i}|}{2\pi}\int_D\ln|x-y||\omega_i(y)|dy\geq |\kappa_{\varepsilon,i}|\mu_{\omega,i}-C.
\end{equation}
In view of Lemma \ref{s1mp7}, we deduce from \eqref{uyiii} that 
\begin{equation}\label{s2mp}
|\kappa_{\varepsilon,i}|\int_D\ln\frac{\varepsilon}{|x-y|}|\omega_i(y)|dy\geq -C.
\end{equation}

Let $R>1$ be a positive number to be determined. We  write the left-hand side of \eqref{s2mp} as follows:
\begin{equation}\label{s5mp}
\begin{split}
&|\kappa_{\varepsilon,i}|\int_D\ln\frac{\varepsilon}{|x-y|}|\omega_i(y)|dy \\
=& |\kappa_{\varepsilon,i}|\int_{B_{R\varepsilon}(x)}\ln\frac{\varepsilon}{|x-y|}|\omega_i(y)|dy  + |\kappa_{\varepsilon,i}|\int_{D\setminus B_{R\varepsilon}(x)}\ln\frac{\varepsilon}{|x-y|}|\omega_i(y)|dy \\
:=& I+II.
\end{split}
\end{equation}
For $I$, applying Lemma \ref{rri1} we have that 
\begin{equation}\label{s7mp}
\begin{split}
I&\leq  |\kappa_{\varepsilon,i}|\int_{ B_{\varepsilon}(\mathbf 0)}\ln\frac{\varepsilon}{|y|}\Pi^i_\varepsilon(y) dy \\
&\leq   |\kappa_{\varepsilon,i}|M\varepsilon^{-2}\int_{  B_{\varepsilon}(\mathbf 0)}\ln\frac{\varepsilon}{|y|} dy \\
&=  |\kappa_{\varepsilon,i}|M \int_{  B_{1}(\mathbf 0)}\ln\frac{1}{|y|} dy \\
&\leq C
\end{split}
\end{equation}
For $II$, we have that 
\begin{equation}\label{s8mp}
II\leq |\kappa_{\varepsilon,i}|\ln\frac{1}{R}\int_{D\setminus B_{R\varepsilon}(x)}|\omega_i(y)|dy.
\end{equation}
 From \eqref{s2mp}-\eqref{s8mp}, we obtain
\begin{equation}\label{s9mp}
 |\kappa_{\varepsilon,i}| \int_{D\setminus B_{R\varepsilon}(x)}|\omega_i(y)|dy \leq \frac{C}{\ln R}.
\end{equation}
From \eqref{s9mp}, we are able to choose a sufficiently large $R$, not depending on $\varepsilon,$ such that
\begin{equation}\label{s10mp}
\int_{D\setminus B_{R\varepsilon}(x)}|\omega_i(y)|dy\leq \frac{1}{3}|\kappa_{ \varepsilon,i}|,
\end{equation}
which is equivalent to
\begin{equation}\label{s11mp}
\int_{ B_{R\varepsilon}(x)}|\omega_i(y)|dy\geq \frac{2}{3}|\kappa_{ \varepsilon,i}|.
\end{equation}

Note that \eqref{s11mp} holds for any $x\in A_{\omega,i}$. Based on this fact, we can show that
\begin{equation}\label{s19mp}
 {\rm diam}(A_{\omega,i})\leq 2R\varepsilon.
\end{equation}
In fact, if \eqref{s19mp} is false, then we can choose two points $x_1,x_2\in A_{\omega,i}$ such that $|x_1-x_2|>2R\varepsilon,$ then
\begin{equation}\label{s1019mp}
\int_D|\omega_i| dx\geq \int_{B_{R\varepsilon}(x_1)}|\omega_i| dx+ \int_{B_{R\varepsilon}(x_2)}|\omega_i| dx\geq \frac{4}{3}|\kappa_{ \varepsilon,i}|,
\end{equation}
which is a contradiction to the fact that
\[\int_D|\omega_i|dx=|\kappa_{\varepsilon,i}|.\]
Choosing $R_0=2R$ we complete the proof.

\end{proof}
\subsection{Limiting location of the vortex cores}

\begin{proposition}\label{proop3}
For any $\delta>0$, there exists some $\varepsilon_2\in(0,\bar r)$, depending only on $D,(\bar x_1,\cdot\cdot\cdot,\bar x_k), \Pi^1,\cdot\cdot\cdot\Pi^k, \vec\kappa$ and $\delta,$ such that for any $\varepsilon\in(0,\varepsilon_2)$, it holds that
\[A_{\omega,i}\subset \overline{B_{\delta}(\bar x_i)},\quad\forall\,\omega\in\mathcal M_{\varepsilon},\,i\in\{1,\cdot\cdot\cdot,k\}.\]

\end{proposition}
\begin{proof}
It suffices to show that for any sequence $\{\varepsilon_n\}_{n=1}^{+\infty}\subset(0,\bar r)$ satisfying $\varepsilon_n\to 0$ as $n\to+\infty,$ any sequence $\{\omega_n\}_{n=1}^{+\infty}$ with $\omega_n\in\mathcal M_{ \varepsilon_n}$, and any sequence $\{x_{n,i}\}_{n=1}^{+\infty}$ with $x_{n,i}\in A_{\omega_n,i}$, it holds that
\[\lim_{n\to+\infty}x_{n,i}=\bar x_i,\quad\forall\,i\in\{1,\cdot\cdot\cdot,k\}.\]
Since   $A_{\omega_n,i}\subset\overline{B_{\bar r}(\bar x_i)}$,
we can assume up to a subsequence that $x_{n,i}\to \hat x_i$ for some $\hat x_i\in \overline{B_{\bar r}(\bar x_i)}$ as $n\to+\infty.$
Define
\[w_n=\sum_{i=1}^k{\rm sgn}(\kappa_i)\Pi^i_{\varepsilon_n}(\cdot-\bar x_i).\]
It is easy to check that $w_n\in \mathcal K_{\varepsilon_n}.$ Hence
\begin{equation}\label{enco1}
E(\omega_n)\geq E(w_n).
\end{equation}
To make the  remaining proof  clear,  we introduce the following notation:
\[\mathcal F(v)=-\frac{1}{4\pi} \sum_{i=1}^k\int_{B_{\bar r}(\bar x_i)}\int_{B_{\bar r}(\bar x_i)}\ln|x-y|v(x)v(y)dxdy,\]
\[\mathcal G(v)=-\frac{1}{4\pi} \sum_{1\leq i\neq j\leq k} \int_{B_{\bar r}(\bar x_i)}\int_{B_{\bar r}(\bar x_j)}\ln|x-y|v(x)v(y)dxdy,\]
\[\mathcal H(v)=-\frac{1}{2}\sum_{1\leq i,j\leq k}\int_{B_{\bar r}(\bar x_i)}\int_{B_{\bar r}(\bar x_j)}h(x,y)v(x)v(y)dxdy,\]
 Then \eqref{enco1} can be written as follows
\begin{equation}\label{enco2}
\mathcal F(\omega_{n})+\mathcal G(\omega_{n})+\mathcal H(\omega_{n})\geq \mathcal F(w_{n})+\mathcal G(w_{n})+\mathcal H(w_{n}).
\end{equation}
Noticing that $w_n$ is radially symmetric and nonincreasing with respect to $\bar x$, we can apply Lemma \ref{rri2} to obtain for any $1\leq i\leq k$
\[
-\frac{1}{4\pi}\int_{B_{\bar r}(\bar x_i)}\int_{B_{\bar r}(\bar x_i)}\ln|x-y|w_n(x)w_n(y)dxdy
\ge -\frac{1}{4\pi}\int_{B_{\bar r}(\bar x_i)}\int_{B_{\bar r}(\bar x_i)}\ln|x-y|\omega_n(x)\omega_n(y)dxdy,
\]
which implies
\begin{equation}\label{enco6}
\mathcal F(\omega_{n})\leq \mathcal F(w_{n}).
\end{equation}
From \eqref{enco2} and \eqref{enco6} we get
\begin{equation}\label{enco7}
\mathcal G(\omega_{n})+\mathcal H(\omega_{n})\geq \mathcal G(w_{n})+\mathcal H(w_{n}).
\end{equation}
To proceed, recalling Proposition \ref{proop2}, it is easy to  check that  for any $1\leq i,j\leq k$ and $\varphi\in C( \overline{B_{\bar r}(\bar x_i)}\times \overline{B_{\bar r}(\bar x_j)})$,
\begin{equation}\label{discc}
\lim_{n\to+\infty}\int_{B_{\bar r}(\bar x_i)}\int_{B_{\bar r}(\bar x_i)}\varphi(x,y)\omega_n(x)\omega_n(y)dxdy= \kappa_i\kappa_j\varphi(\hat x_i,\hat x_j).
\end{equation}
Similarly, for any $1\leq i,j\leq k$ and $\varphi\in C( \overline{B_{\bar r}(\bar x_i)}\times \overline{B_{\bar r}(\bar x_j)})$, it holds that
\begin{equation}\label{discc2}
\lim_{n\to+\infty}\int_{B_{\bar r}(\bar x_i)}\int_{B_{\bar r}(\bar x_i)}\varphi(x,y)w_n(x)w_n(y)dxdy= \kappa_i\kappa_j\varphi(\bar x_i,\bar x_j).
\end{equation}
Using \eqref{discc}, we have
\begin{equation}\label{enco71}
\lim_{n\to+\infty}\mathcal G(\omega_{n})=-\frac{1}{4\pi}\sum_{1\leq i\neq j\leq k}\kappa_i\kappa_j\ln|\hat x_i-\hat x_j|,
\end{equation}
\begin{equation}\label{enco72}
\lim_{n\to+\infty}\mathcal H(\omega_{n})=-\frac{1}{2}\sum_{1\leq i,j\leq k}\kappa_i\kappa_jh(\hat x_i,\hat x_j).
\end{equation}
Using \eqref{discc2}, we have
\begin{equation}\label{enco73}
\lim_{n\to+\infty}\mathcal G(w_{n})=-\frac{1}{4\pi}\sum_{1\leq i\neq j\leq k}\kappa_i\kappa_j\ln|\bar x_i-\bar x_j|,
\end{equation}
\begin{equation}\label{enco74}
\lim_{n\to+\infty}\mathcal H(w_{n})=-\frac{1}{2}\sum_{1\leq i,j\leq k}\kappa_i\kappa_jh(\bar x_i,\bar x_j).
\end{equation}
From \eqref{enco71}-\eqref{enco74}, we are able to pass to the limit $n\to+\infty$ in \eqref{enco7} to get
\begin{equation}\label{enco75}
\begin{split}
 &-\frac{1}{4\pi}\sum_{1\leq i\neq j\leq k}\kappa_i\kappa_j\ln|\hat x_i-\hat x_j|-\frac{1}{2}\sum_{1\leq i,j\leq k}\kappa_i\kappa_jh(\hat x_i,\hat x_j)\\
 \geq&-\frac{1}{4\pi}\sum_{1\leq i\neq j\leq k}\kappa_i\kappa_j\ln|\bar x_i-\bar x_j|-\frac{1}{2}\sum_{1\leq i,j\leq k}\kappa_i\kappa_jh(\bar x_i,\bar x_j).
\end{split}
\end{equation}
By the definition of the Kirchhoff-Routh function $ W$ (see  \eqref{krf}),   \eqref{enco75} can be written equivalently as
\[W(\hat x_1,\cdot\cdot\cdot,\hat x_k)\leq W(\bar x_1,\cdot\cdot\cdot,\bar x_k).\]
Since we have assumed that $(\bar x_1,\cdot\cdot\cdot,\bar x_k)$ is the unique minimum point of $W$ in $\overline{B_{\bar r}(\bar x_1)}\times\cdot\cdot\cdot\times \overline{B_{\bar r}(\bar x_k)}$ (see \eqref{br3}), we obtain
\[(\hat x_1,\cdot\cdot\cdot,\hat x_k)=(\bar x_1,\cdot\cdot\cdot,\bar x_k).\]
 Hence the proof is finished.

\end{proof}

\subsection{Proof of Theorem \ref{thmex}}Now we are ready to complete the proof of Theorem \ref{thmex}.
\begin{proof}[Proof of Theorem \ref{thmex}]
It suffices to show that $\omega\in\mathcal M_{ \varepsilon}$ is a steady weak solution to the vorticity equation if $\varepsilon$ is sufficiently small. By Proposition \ref{proop3}, we see that
\[{\rm dist(supp}(\omega),\partial B_{\bar r}(\bar x))>0,\quad \forall\,\omega\in\mathcal M_{ \varepsilon}\]
when $\varepsilon$ is   small. The desired result follows from Lemma A in Section 1.

\end{proof}

\section{Concentrated vortex flows with prescribed profile functions}

In this section, we discuss the stability of a class of concentrated vortex flows with prescribed profile functions.
To begin with, we state the following existence result proved by Cao-Wang-Zhan in \cite{CWZ}.

Let $k$ be a positive integer, $\kappa_1,\cdot\cdot\cdot,\kappa_k$ be $k$ nonzero real numbers. Let $(\bar x_1,\cdot\cdot\cdot,\bar x_k)$ be an isolated local minimum point of the Kirchhoff-Routh function $W$ related to $\vec\kappa$. Choose $\bar r>0$ sufficiently small such that
\eqref{br1}-\eqref{br3} hold.

Let  $\varepsilon, \Lambda$ be two positive parameters. Define
\begin{equation*}
\mathcal{A}_{\varepsilon,\Lambda}=\left\{w=\sum_{i=1}^kw_i\,\,\bigg|\,\,w_i\in L^\infty(D), \mbox{supp($w_i$)}\subset \overline{B_{\bar r}(\bar{x}_i)}, 0\le \mbox{sgn}(\kappa_i)w_i \le \frac{\Lambda}{\varepsilon^2}, \int_{D}w_idx=\kappa_i\right\}.
\end{equation*}
It is easy to see that $\mathcal A_{\varepsilon,\Lambda}$ is nonempty if
\begin{equation}\label{mane}
\Lambda>\max\left\{1,\frac{\varepsilon^2|\kappa_1|}{\pi\bar r^2},\cdot\cdot\cdot,\frac{\varepsilon^2|\kappa_k|}{\pi\bar r^2}\right\}.
\end{equation}

Let  $\mu_0,\tau_0$ be fixed positive numbers. We impose the following conditions on the profile function $f$:
\begin{itemize}
\item[(C1)] $f\in C(\mathbb R)$, $f(s)=0$ if  $s\leq 0$, and $f$ is strictly increasing on $[0,+\infty)$;
\item[(C2)] there exists $\mu_0\in(0,1)$ such that
\[\int_0^sf(r)dr\leq \mu_0f(s)s,\,\,\forall\,s\geq0;\]
\item[(C3)] for all $\tau_0>0$, it holds that
\[\lim_{s\to+\infty}f(s)e^{-\tau_0 s}=0.\]
\end{itemize}
For $f$ satisfying (C1)(C2)(C3), we define a new function $H_f$ as follows

\begin{equation}\label{p2pp}
H_f(s)=\int_0^sh(r)dr,
\end{equation}
  where
  \begin{equation}
  h(s)=
  \begin{cases}
  f^{-1}(s),&\mbox{ if }s>0;\\
  0,&\mbox{ if }s\leq0.
  \end{cases}
  \end{equation}

Let $f_1,\cdot\cdot\cdot,f_k$ be $k$ real functions satisfying (C1)(C2)(C3). Define
$$\mathcal{E}(\omega)=\frac{1}{2}\int_D \omega(x)\mathcal{G}\omega(x)dx-\frac{1}{\varepsilon^2}\sum_{i=1}^k\int_D H_{f_i}(\varepsilon^2\mbox{sgn}(\kappa_i)\omega\mathbf 1_{_{B_{\bar r}(\bar{x}_i)}})dx.$$
Denote  $\mathcal N_{\varepsilon,\Lambda}$ the set of maximizers of $\mathcal E$ over $\mathcal{A}_{\varepsilon,\Lambda}$.

The following existence theorem has been proved in \cite{CWZ}.
\begin{theorem}[\cite{CWZ}]\label{cwzthm}
There exists $\varepsilon_0,\Lambda_0>0$, depending only on $D, (\bar x_1,\cdot\cdot\cdot,\bar x_k),\vec\kappa,\mu_0,\tau_0$, such that for any $0<\varepsilon<\varepsilon_0$ and $\Lambda>\Lambda_0$, $\mathcal N_{\varepsilon,\Lambda}$ is nonempty, and any $\omega\in\mathcal N_{\varepsilon,\Lambda}$ satisfies
\[\omega=\frac{1}{\varepsilon^2} sgn(\kappa_i)f_i\big(\mbox{sgn}(\kappa_i)\mathcal{G}\omega-\mu_{\omega,i}\big) \mbox{ a.e. in }B_{\bar r}(\bar x_i),\,\,i=1,\cdot\cdot\cdot,k,\]
where each $\mu_{\omega,i}$ is a real number depending on $\omega$. Moreover, for fixed $\Lambda>\Lambda_0$, we have
\begin{itemize}
\item[(i)] $\mu_{\omega,i}\geq -\frac{|\kappa_i|}{2\pi}\ln\varepsilon-C$ for any $\omega\in \mathcal N_{\varepsilon,\Lambda}$ and $i=1,\cdot\cdot\cdot,k$, where $C$ is a positive number depending only on $D, (\bar x_1,\cdot\cdot\cdot,\bar x_k),\vec\kappa,\mu_0,\tau_0, \Lambda;$
\item[(ii)] diam(supp($\omega$))$\leq R_0\varepsilon$ for any $\omega\in \mathcal N_{\varepsilon,\Lambda}$, where $R_0$ is a positive number depending only on  $D, (\bar x_1,\cdot\cdot\cdot,\bar x_k),\vec\kappa,\mu_0,\tau_0,\Lambda$;
\item[(iii)] for any $\delta>0,$ there exists some $\varepsilon_1>0,$ depending only on  $D,(\bar x_1,\cdot\cdot\cdot,\bar x_k),\vec\kappa,\mu_0,\tau_0,\Lambda,\delta$, such that for any $\varepsilon\in(0,\varepsilon_1)$, it holds that
\[\mbox{supp}(\omega\mathbf 1_{B_{\bar r}(\bar x_i)})\subset \overline{B_{\delta}(\bar{x}_i)},\quad \forall\,\omega\in\mathcal N_{\varepsilon,\Lambda},\]
and consequently by Lemma A any $\omega\in\mathcal N_{\varepsilon,\Lambda}$ is a steady weak solution to the vorticity equation.
\end{itemize}

\end{theorem}

Unlike the variational problem with prescribed rearrangement of vorticity in Section 1, here we do not know
whether the elements in $\mathcal N_{\varepsilon,\Lambda}$ have the same rearrangements, and it is also hard to prove the compactness of $\mathcal N_{\varepsilon,\Lambda}$. Hence we can not apply Theorem \ref{thm551} directly to obtain stability.

However, the good thing here is that we know what the profile functions are, which means that the stream function satisfies a definite semilinear elliptic equation. This allows us to employ the methods from the field of elliptic equations to obtain fine estimates for the solutions.

From the viewpoint of the stream function, recently  Cao-Yan-Yu-Zou \cite{CYYZ} proved the following local uniqueness result when each $f_i$ is a power function.

\begin{theorem}[Local uniqueness]\label{loun}
Let $q\in(0,+\infty)$ be fixed.
In the setting of Theorem \ref{cwzthm}, Suppose additionally
\begin{itemize}
\item[(i)]$(\bar x_1,\cdot\cdot\cdot,x_k)$ is  a nondegenerate critical point of $D$;
\item [(ii)] $\kappa_i>0$ for any $i=1,\cdot\cdot\cdot,k$;
\item[(iii)] $f_i(s)=s_+^q$ for any $i=1,\cdot\cdot\cdot,k$.
\end{itemize}
Then for fixed $\Lambda>\Lambda_0,$ there exists some $\varepsilon_2>0$, depending only on  $D,(\bar x_1,\cdot\cdot\cdot,\bar x_k),\vec\kappa,\Lambda$, such that for any $\varepsilon\in(0,\varepsilon_2),$ $\mathcal N_{\varepsilon,\Lambda}$ is a singleton $\{\omega_{\varepsilon,\Lambda}\}$.

\end{theorem}

Based on Theorem \ref{loun}, we can proved the following stability result for $k=1$.

\begin{theorem}[Stability]\label{sres}
In the setting of Theorem \ref{loun}, if $k=1$ and $\bar x\in D$ is a nondegenerate local minimum point of $H$, then for fixed $\Lambda>\Lambda_0,$ there exists some $\varepsilon_3>0$, depending only on  $D,\bar x,\kappa,\Lambda$, such that for any $\varepsilon\in(0,\varepsilon_3),$ $\omega_{\varepsilon,\Lambda}$ is  stable in the $L^p$ norm of the vorticity with respect to initial perturbations in $L^\infty(D)$ for any $p\in(1,+\infty)$.
\end{theorem}
\begin{proof}
 By Theorem \ref{thm551} in Section 3 or Theorem 5 in \cite{B5}, it suffices to show that $\omega_{\varepsilon,\Lambda}$ is an isolated maximizer of $E$ over $\mathcal R_{\omega_{\varepsilon,\Lambda}}$. In view of Theorem \ref{loun}, it is easy to see that $\omega_{\varepsilon,\Lambda}$ is the unique  maximizer of $E$ over
\[\left\{w\in L^\infty(D)\,\,\big|\,\, w\in\mathcal R_{\omega_{\varepsilon,\Lambda}},\,\,{\rm supp}(w)\subset \overline{B_{\bar r}(\bar x)}\right\}\]
if $\varepsilon$ is sufficiently small. Taking into account (iii) in Theorem \ref{cwzthm}, we can easily get  isolatedness by repeating the argument as in Section 4.
\end{proof}
\begin{remark}
Since $H(x)\to +\infty$ as $x\to\partial D$,   $H$ attains its global minimum at some point $\bar x\in D$.
If $D$ is additionally convex, then  $H$ is strictly convex by \cite{CF}, and thus  $\bar x$ must be unique and nondegenerate in this case.
\end{remark}

\begin{remark}
If local uniqueness holds for $k=2$ with $ \kappa_1\kappa<0$, then the corresponding stability result as in Theorem \ref{sres} holds as well.
\end{remark}

\bigskip
\noindent{\bf Acknowledgements.}
We are very grateful to the   anonymous referee for useful comments and suggestions.
{G. Wang was supported by National Natural Science Foundation of China (12001135, 12071098) and China Postdoctoral Science Foundation (2019M661261, 2021T140163).}

\bigskip

\noindent{\bf Compliance with ethical standards}

\bigskip
\noindent{\bf Conflict of interest}  The authors declare that they have no conflict of interest to this work.

\phantom{s}
 \thispagestyle{empty}

\end{document}